\documentclass[12pt, colorinlistoftodos]{amsart}
\usepackage{a4wide}
\usepackage{amssymb}
\usepackage{amsthm}
\usepackage{amsmath}
\usepackage{mathtools}
\usepackage{physics}
\usepackage{enumitem}
\usepackage{amscd}
\usepackage{verbatim}
\usepackage{bm}
\usepackage{soul}
\usepackage[mathscr]{eucal}
\usepackage[all]{xy}
\usepackage{todonotes}
\allowdisplaybreaks 
\usepackage[toc,page]{appendix}

\usepackage{dsfont, stmaryrd}
\usepackage{tikz, tikz-cd, caption, tabu}
\usepackage{tkz-euclide}
\usepackage[autostyle,italian=guillemets]{csquotes}
\usepackage[style=alphabetic,sorting=nyt,firstinits=true,maxnames=5,maxalphanames=5,backend=biber,doi=false, url=false, isbn=false]{biblatex}
\usepackage{hyperref}

\numberwithin{equation}{section}

\theoremstyle{plain}
\newtheorem{theorem}{Theorem}[section]
\newtheorem{corollary}[theorem]{Corollary}
\newtheorem{lemma}[theorem]{Lemma}
\newtheorem{proposition}[theorem]{Proposition}

\theoremstyle{definition}
\newtheorem{definition}[theorem]{Definition}
\newtheorem{remark}[theorem]{Remark}

\theoremstyle{remark}

\newcommand{\OO}{\mathcal O}
\newcommand{\A}{\mathbb{A}}
\newcommand{\R}{\mathbb{R}}
\newcommand{\Q}{\mathbb{Q}}
\newcommand{\Z}{\mathbb{Z}}
\newcommand{\N}{\mathbb{N}}
\newcommand{\C}{\mathbb{C}}
\newcommand{\h}{\mathbb{H}}
\renewcommand{\H}{\mathbb{H}}

\newcommand{\vol}{\mathrm{vol}}

\usepackage{subfiles}

\newcommand{\indco}{\delta}

\newcommand{\coef}{c}
\newcommand{\indcoef}{\nu}
\newcommand{\opcomf}{\widehat{\Gamma}^\prime}
\newcommand{\conmu}{\xi}
\newcommand{\qform}{q}
\newcommand{\ind}{r}

\newcommand{\genU}{e}
\newcommand{\genUU}{e'}

\newcommand{\indone}{\alpha} 
\newcommand{\indtwo}{{\alpha'}} 
\newcommand{\coeff}{y} 
\newcommand{\dual}[1]{#1 ^\vee} 

\newcommand{\KMlift}{{\Lambda_{\operatorname{KM}}}}
\newcommand{\KMliftL}{{\Lambda^L_{\operatorname{KM}}}}
\newcommand{\Hom}{\mathrm{Hom}}

\newcommand{\genH}{e} 
\newcommand{\genHH}{e'} 
\newcommand{\varphiKM}{\varphi_{\rm KM}}
\newcommand{\basevec}{e} 
\newcommand{\baseco}{y} 
\newcommand{\Gpol}{Q} 

\newcommand{\genvec}{x} 
\newcommand{\HH}{\H}
\newcommand{\RR}{\R}
\newcommand{\QQ}{\Q}
\newcommand{\ZZ}{\Z}
\newcommand{\CC}{\C}

\newcommand{\bigO}{\mathrm{O}}
\newcommand{\domain}{\mathcal{D}}
\newcommand{\tubedom}{\mathcal{H}}
\newcommand{\Pet}{\mathrm{Pet}}
\newcommand{\NP}{N_P}
\newcommand{\MP}{M_P}
\newcommand{\nP}{n_P}
\newcommand{\mP}{m_P}
\newcommand{\Cl}{\mathrm{Cl}}
\newcommand{\dis}{D} 

\newcommand{\lp}{\left (}
\newcommand{\rp}{\right )}

\newcommand{\Fc}{{\mathcal{F}}}
\newcommand{\Oc}{{\mathcal{O}}}
\newcommand{\Sc}{{\mathcal{S}}}
\newcommand{\ec}{{\mathfrak{e}}}
\newcommand{\Zb}{\mathbb{Z}}

\newcommand{\Qb}{\mathbb{Q}}
\newcommand{\SO}{{\mathrm{SO}}}
\newcommand{\af}{\mathfrak{a}}
\newcommand{\df}{\mathfrak{d}}
\newcommand{\ef}{\mathfrak{e}}
\newcommand{\Nm}{{\mathrm{Nm}}}
\newcommand{\Ab}{\mathbb{A}}
\newcommand{\Hb}{\mathbb{H}}
\newcommand{\ebf}{{\mathbf{e}}}
\newcommand{\attr}[1]{\color{red} {#1} \color{black}}
\renewcommand{\tr}{\operatorname{Tr}}
\newcommand{\Sp}{\mathrm{Sp}}
\newcommand{\e}{\mathbf{e}}
\newcommand{\SL}{\mathrm{SL}}
\renewcommand{\aa}{a}

\newcommand{\riccardo}[1]{\textcolor{black}{#1}}
\newcommand{\yl}[1]{\todo[color=yellow]{YL: #1}}

\begin{document}
\title[An Adelic Approach to the Kudla--Millson Lifts over Totally Real Fields]{An Adelic Approach to the Kudla--Millson Lifts over Totally Real Number Fields}

\author[Y.~Li]{Yingkun Li}
\address{
  School of Mathematics,
  Institute of Advanced Studies,
1 Einstein Drive, 
Princeton NJ 08540, USA}
\email{lykpi@ias.edu}

\author[M.~Zhang]{Mingkuan Zhang}
\address{
Fachbereich Mathematik, Technische Universität Darmstadt, Schlossgartenstraße 7, D–
64289 Darmstadt, Germany.
}
\email{mzhang@mathematik.tu-darmstadt.de}

\author[R.~Zuffetti]{Riccardo Zuffetti}
\address{
Fachbereich Mathematik, Technische Universität Darmstadt, Schlossgartenstraße 7, D–
64289 Darmstadt, Germany.
}
\email{zuffetti@mathematik.tu-darmstadt.de}

\subjclass[2020]{}
\thanks{
}

\maketitle


\begin{abstract}
We compute the Fourier expansions of the Kudla–Millson theta lifts of vector-valued Hilbert cusp forms of parallel weight.
These expansions are with respect to $0$-dimensional cusps of non-compact orthogonal Shimura varieties defined over any totally real number field, and computed using a mixed model of the adelic Weil representation under a partial Fourier transform.
As a consequence, we obtain the injectivity of the theta lifts under some additional hypothesis on the variety.
Furthermore, we show that the theta lifts are square-integrable harmonic differential forms, and deduce lower bounds for dimensions of cohomology groups of such varieties.
\end{abstract}

 \makeatletter
 \providecommand\@dotsep{5}
 \def\listtodoname{List of Todos}
 \def\listoftodos{\@starttoc{tdo}\listtodoname}
 \makeatother

\allowdisplaybreaks
\tableofcontents


\section{Introduction}

Kudla and Millson constructed, for orthogonal and unitary locally symmetric spaces, distinguished Schwartz forms whose associated theta series are (non-holomorphic) Siegel modular forms with values in closed differential forms on the symmetric space~\cite{kudlamillson-harmonicI} \cite{kudlamillson-harmonicII}.
The cohomology classes of such Kudla--Millson theta series are holomorphic modular forms~\cite{KM-intersections}.
Their geometric significance lies in the fact that their Fourier coefficients realize the Poincaré duals of cycles of Heegner type on the associated locally symmetric spaces.

There are many papers in the literature that investigate the arithmetic and geometric properties of the theta lifts for orthogonal Shimura varieties, constructed by pairing Siegel cusp forms with Kudla--Millson theta functions under the Petersson inner product~\cite{kudlamillson-tubes} \cite{bruinier-habil} \cite{bruinier-funke} \cite{bruinier-funke-inj} \cite{bruinier14}  \cite{stein-converse} \cite{zuffetti_unfolding1} \cite{bruinier-zuffetti} \cite{romain} \cite{metzler} \cite{kieferzuffetti}  \cite{metzlerzuffetti}. 
These results, together with the \emph{injectivity} of Kudla--Millson lifts, have many nice consequences, such as converse theorems for Borcherds products, dimension formulas for cohomology groups, and geometric properties of cones generated by special cycles in the pseudoeffective cones of orthogonal Shimura varieties \cite{bruiniermoller} \cite{zuf-manu} \cite{zuf-tran} \cite{bbfw} \cite{bfz}.

All the results cited so far have been established in the setting of orthogonal Shimura varieties arising from quadratic spaces \emph{defined over $\mathbb{Q}$}.
In the present paper, we investigate the Kudla--Millson theta lift for orthogonal Shimura varieties associated to quadratic spaces defined over a \emph{totally real number field~$F$}.
In this context, the automorphic input consists of \emph{Hilbert cusp forms} of parallel weight.
Since the base field is an extension of~$\QQ$, we adopt an adelic formulation, which allows for a uniform treatment of all archimedean places simultaneously.


\subsection{The theta lifts}
To describe in more details the results of the present paper, we introduce some notations.
Let $G=\SL_2$ be defined over $F$ and let~$d=[F:\QQ]$.
We denote by~$\Ab,\Ab_{f},\Oc,\df$ and $\dis$, the adeles, finite adeles, ring of integers, different ideal, and discriminant of~$F$ respectively, 
and define
$F_\infty\coloneqq F\otimes_\mathbb{Q}\mathbb{R}, \widehat{\mathcal{O}} \coloneqq 
\prod_\mathfrak{p}\Oc_{\mathfrak{p}}$.

Let $V$ be an isotropic, $F$-quadratic space of dimension $n+2$ for $n\in \N$, with quadratic form $q$ such that~$V$ has signature~$(n,2)$ at all real places of $F$. 
Then the Hermitian symmetric space  $\mathcal{D}$ associated to~$H=\SO(V)$ has complex dimension~$dn$.
Consider an even $\mathcal{O}$-lattice~$L\subset V$.
For clarity, throughout this introduction we assume that~$L$ is $\Zb$-\emph{unimodular}, i.e.\ the restriction of scalar of $L$ to $\Zb$ is unimodular\footnote{This in particular implies that $n$ is even \cite{BL23}.}.
This simplifying assumption is removed in the main body of the paper, where we work in full generality.

For an open compact subgroup $K \subset H(\Ab_f)$ that stabilizes $\widehat{L}\coloneqq L\otimes_{\Oc}\widehat{\mathcal{O}}$,
the following double quotient consists of
the complex points of an orthogonal Shimura variety
\begin{align*}
  X_K=
H(F)  \backslash\mathcal{D}\times H(\Ab_f)/K.
\end{align*}
Since $V$ is isotropic,  $X_K$ is non-compact.
The variety $X_K$ is in general disconnected, and the number of connected components is at least the size of the class group of $F$
\begin{equation} \label{eq:Cl}
\Cl(F) \coloneqq  F^\times \backslash \Ab_f^\times /\widehat{\Oc}^\times.
\end{equation} 
Equality happens when $K$ is maximal. 
We denote by $Y_K \cong \Gamma_K\backslash\domain$ the connected component of $X_K$ containing the identity of $H(\Ab_f)$.

In this setting, the Kudla--Millson theta function arising from the dual reductive pair $(G,H)$ is a map
\[
  \Theta_{\mathrm{KM}}\colon
  G(F)\backslash G(\Ab)\to
  \mathcal{A}^{2d}(X_K),
\]
where~$\mathcal{A}^{2d}(X_K)$ is the space of differential~$2d$-form on~$X_K$.
The de Rham cohomology class~$[\Theta_{\mathrm{KM}}]$ is a holomorphic Hilbert modular form, whose Fourier coefficients are classes of certain algebraic cycles of codimension~$d$ in~$X_K$, see~\cite{rosuyott} and~\cite{kudla-remarks} for details.

Let~$S_{k}$ be the space of \emph{Hilbert cusp forms}~$f\colon\HH^d\to\CC$ of parallel weight~$k=\frac{n+2}{2}$ with respect to~$\SL_2(\mathcal{O})$.
The theta lift associated to~$\Theta_{\mathrm{KM}}$ is defined as
\[
\Lambda_{\mathrm{KM}}\colon S_{k}\to \mathcal{A}^{2d}(X_K),\qquad
f(\tau) = \sum_{\substack{\indcoef \in \mathfrak{d}^{-1,+}}} \coef_\indcoef(f) \e(\indcoef \tau)
\mapsto\int_{G(F)\backslash G(\Ab)}
 \Theta_{\mathrm{KM}}(g) \overline{f^\#(g)}\, dg,
\]
where $\tau \in \h^d, \e (\cdot) = \exp (2\pi i \tr \cdot)$, $f^\#\colon G(F) \backslash G(\A)\to\CC$ is the adelization of~$f$
and~$dg$ is a suitably normalized Haar measure on~$G(\A)$, see Sections~\ref{sec:groupvarexpl} and \ref{sec:KM} for more information.
If~$F=\QQ$, then~$\Lambda_{\mathrm{KM}}$ is simply an adelic rewriting of the classical Kudla--Millson lift.
In general, the Hilbert modular variety for $G$ will still be connected because of strong approximation for $\SL_2$. However, the number of cusps of $\SL_2(\Oc)\backslash \Hb^d$ will be $|\Cl(F)|$.



\subsection{Fourier expansions of theta lifts}
The first result of the present paper is the Fourier expansion of~$\Lambda_{\mathrm{KM}}(f)$, for~$f\in S_{k}$, with respect to any $0$-dimensional cusp of~$Y_K$.
The choice of such a cusp is equivalent to the choice of an isotropic line~$\ell$ in~$V$.
The Fourier expansion is described in terms of the tube-domain realization~$\mathcal{H}$ of~$\domain$ with respect to~$\ell$, on which the parabolic subgroup~$P(F_\infty)\subset H(F_\infty)$ stabilizing~$\ell$ acts transitively.
We fix a base point $Z_0 \in \mathcal{H}$ as in section \ref{sec:tubedomains}.
For every~$Z=X+iY\in\mathcal{H}$, we denote by $h_Z
\in P(F_\infty)$ the standard element such that $h_Z \cdot Z_0 = Z$
(see \eqref{eq;explchofh_Z}).

Let $U_{}=\Oc \oplus \df^{-1}$ be the hyperbolic plane defined as the $\Zb$-unimodular lattice endowed with the quadratic form~$q_{U_{}} (x,y) \coloneqq xy$.
In general, hyperbolic planes are given by an element in $\A_f^\times$, and their isomorphism class is determined by a class in $\Cl(F)$ (see Section~\ref{sec:KM}). 
In the main body of the paper, we work with these more general hyperbolic planes.
For the subspace $V_0 \coloneqq  U^\perp \subset V$, denote the character $\chi_{V_0}(x) = (x, (-1)^{\frac{n(n-1)}{2}} \det V_0)_F,\ x \in \Ab^\times/F^\times$.
We denote by~$\zeta_F(s)$ the Dedekind zeta function for~$F$, and $c_\nu(f, \xi)$ an integral related to the Fourier coefficient of $f$ at $\xi\in \Cl(F)$ as in \eqref{eq:FC-xi}.
The Kudla--Millson lifts of Hilbert cusp forms have the following Fourier expansion, see Theorem~\ref{Fourier expansion} for a more general result.

\begin{theorem}\label{thm:intro}
Let $f\in S_{k}$ be a Hilbert cusp form of parallel weight $k=\frac{n+2}{2}$. 
Suppose that~$L = L_0 \oplus U$ is $\Zb$-unimodular. 
Then the Fourier expansion of~$\KMliftL(f)$ with respect to the cusp corresponding to the isotropic line~$\Oc \subset U$ on $Y_K \subset X_K$ is  
\begin{align*} 
\begin{split}
  \KMliftL(f)(Z) &=
    \KMliftL_{, 0}(f)(Y,1) 
    +
    \frac{ (2\pi)^d }{    \dis \zeta_F(2)}
    \sum_{\xi \in \Cl(F)}
        |\xi|_\A^{\frac{n}{2}} \chi_{V_0} (\xi)
\sum_{\delta \in (\xi \hat\Oc \cap F^\times)/\Oc^\times}    | \Nm(\indco)|^{n-1}\\
&\quad\times    \sum_{\substack{ x\in (L_0 \otimes F \cap \delta \xi^{-1}\widehat{L_0}) }} \e\lp (X,x)\rp
    \overline{\coef_{q(x /\indco)} (f_{},\xi)}
  \sum_{\underline{\indone},\underline{\indtwo}}
  W_{(\underline{\indone},\underline{\indtwo})}   (Y, x )
      (h_Z^{-1})^*\omega_{\underline{\indone},\underline{\indtwo}}. 
\end{split}
\end{align*}
Here $\KMliftL_{ ,0}(f)(Y,1)$ is the constant term defined in~\eqref{decomposed theta lift}, $  W_{(\underline{\indone},\underline{\indtwo})}   (Y, x )$ is the auxiliary function~\eqref{eq:Wx0}, and~$\omega_{\underline{\indone},\underline{\indtwo}}\in\wedge^{2d}T_{Z_0}^*\mathcal{H}$ are the cotangent vectors~\eqref{eq:omega} defining the Kudla--Millson theta form. 
\end{theorem}

The key ingredient to prove Theorem~\ref{thm:intro} is to rewrite the theta kernel, which is defined in terms of the Schrödinger model of the Weil representation~$\omega$ of $G(\A)\times H(\A)$, with respect to the mixed model of~$\omega$ induced by the isotropic line~$\Oc \subset U$.
This is achieved by computing partial Fourier transforms of the Kudla--Millson Schwartz form defining~$\Theta_{\mathrm{KM}}$.

As a consequence of Theorem~\ref{thm:intro}, we deduce the following injectivity criterion, see Theorem \ref{injectivity} and \ref{injectivity two hyper} for general results.
\begin{theorem}
  \label{thm:inj-intro}
  If~$L$ splits off~$U\oplus U$, then the Kudla--Millson lift~$\KMlift$ as well as its restriction to~$Y_K$ are injective.
\end{theorem}

\subsection{Comparison with Borcherds' formalism}

When $F=\QQ$, the Fourier expansion of the Kudla--Millson lift was computed in \cite{zuffetti_unfolding1} using Borcherds’ unfolding method~\cite{borcherds_singularities}.
That approach relies on decomposing Siegel theta functions attached to homogeneous polynomials with respect to the splitting of a hyperbolic plane in~$L$.
While extremely powerful, this formalism is technically involved.

In Section~\ref{sec:comparison} we describe Borcherds’ method in terms of the action of the parabolic subgroup~$P(F_\infty)$ on~$\mathcal H$, and illustrate how Borcherds' formalism boils down to our adelic approach in this case.
As a consequence, we recover the Fourier expansion of~\cite{zuffetti_unfolding1} as a special case of Theorem~\ref{thm:intro}.
Beyond serving as a consistency check, this comparison sheds light on Borcherds’ computations, translating them into the adelic framework.

\subsection{Harmonicity of theta lifts and applications}
In Section~\ref{sec:harmonic} we prove that the Kudla--Millson lifts of Hilbert cusp forms are \emph{harmonic} differential forms with respect to the~$H(F_\infty)$-invariant Laplacian on~$\domain$, generalizing~\cite[Theorem~4.1]{kudlamillson-tubes}.
This is done by a careful study of how the Casimir elements of the universal enveloping Lie algebras of~$G(F_\infty)$ and~$H(F_\infty)$ act on~$\Theta_{\mathrm{KM}}$.
Since in low degree the de Rham and~$L^2$-cohomology groups of~$Y_K$ are isomorphic, as a drect consequence of Theorem \ref{cor;applicinL2}, we deduce the following result.

\begin{theorem}
    Let~$n>3$.
    The dimension of the de Rham cohomology group~$H^{2d}(X_K)$ is bounded below by the dimension of the image of~$\Lambda_{\mathrm{KM}}$.
    In particular, if~$\Lambda_{\mathrm{KM}}$ is injective, then
    \[
    \dim H^{2d}(X_K)\geq\dim S_{k}.
    \]
    The same is true if~$X_K$ is replaced with its connected component~$Y_K$ and~$\Lambda_{\mathrm{KM}}$ is restricted to~$Y_K$.
\end{theorem}

\subsection{Acknowledgments}
We would like to thank Jan Bruinier, Paul Kiefer and Eugenia Rosu for useful conversations on the topics of the present paper,
and thank Yong Hu and Rainer Schulze-Pillot for careful and helpful suggestions on the representing behaviors of integral lattices over number fields. The authors acknowledge support by Deutsche Forschungsgemeinschaft (DFG, German Research Foundation) through the Collaborative Research Centre TRR 326 \emph{Geometry and Arithmetic of Uniformized Structures}, project number 444845124.
Y.~Li is supported by the DFG Heisenberg Program, project number 539345613.
R.~Zuffetti is supported by the DFG Research Grant (Eigene Stelle) \emph{Arithmetic and geometry of the Kudla--Millson theta function}, project number 554793187.

\section{Preliminaries}\label{sec:prelim}
Let~$F$ be a totally real extension of degree~$d$ over~$\QQ$, and let~$\sigma_1,\dots,\sigma_d$ be the different embeddings of~$F$ in~$\RR$.
Let~$V$ be a quadratic space over~$F$ with bilinear form~$(\cdot{,}\cdot)$.
The associated quadratic form is~$q(\cdot)\coloneqq(\cdot{,}\cdot)/2$.
In this paper, we always assume that~$V\otimes_{\sigma_r}\RR$ has signature~$(n,2)$ for every~$r$, and that~$V$ contains isotropic lines $\ell$ and $\ell'$ spanning a hyperbolic plane over~$F$.

\subsection{Decomposition of $V$ with respect to an isotropic line}\label{sec:decompofV}
Consider the decomposition
\begin{equation}\label{eq;decVF}
V=\ell\oplus V_0\oplus \ell',\quad V_0\coloneqq(\ell+\ell')^\perp\subset V.  
\end{equation}
We write every~$x\in V$ with respect to the decomposition~\eqref{eq;decVF} as~$x=x_2\genH + x_0 + x_1\genHH$, where~$\genH$ is a basis vector of~$\ell$ and~$\genHH\in\ell'$ is its dual, namely \riccardo{such that} $(\genH,\genHH)=1$.
We consider~$x$ as the column vector
\begin{equation}\label{eq;columnbas}
x=\left(\begin{smallmatrix} x_2\\ x_0\\ x_1 \end{smallmatrix}\right).
\end{equation}
The bilinear form of~$V$ now can be written as $(x,x)=(x_0,x_0) + x_1x_2$.

We endow the real spaces~$V_\ind\coloneqq V\otimes_{\sigma_\ind}\R$ with the extension of the quadratic form~$q$ with respect to the embedding~$\sigma_\ind$.
We fix \emph{once and for all} a pseudo-orthonormal basis of~$V_\ind$ as follows. 
Let $e_2,\dots,e_{n+1}$ be an orthogonal basis of $V_0$ such that $q(e_{n+1})$ is totally negative.
We define
\begin{align*}
e_1^\ind &= \frac{ \sigma_\ind(\genU)+\sigma_\ind(\genUU)}{\sqrt{2}},\quad 
e_{n+2}^\ind = \frac{\sigma_\ind(\genU)-\sigma_\ind(\genUU)}{\sqrt{2}},\quad\\
e_j^\ind &= \frac{1}{\sqrt{|\sigma_\ind(2q(e_j))|}}\sigma_\ind(e_j),\quad  j =2,\dots,n+1.
\end{align*}
Then the basis $\{ e_j^\ind \}_{j=1}^{n+2}$ is pseudo-orthonormal, i.e.,
\begin{align*}
(e_k^\ind,e_l^\ind)=
\begin{cases}
1 & \text{if~$k=l\in\{1,\dots,n\}$,}\\
-1 & \text{if~$k=l\in\{n+1, n+2\}$,}\\
0 & \text{otherwise.}
\end{cases}
\end{align*}
We write any~$x\in V_\ind$ over this basis as~$\genvec=\sum_j \coeff_j e^\ind_j$.
The elements $\{ e_j^\ind \}_{j=2}^{n+1}$ form a pseudo-orthonormal basis of $V_0\otimes_{\sigma_\ind}\R$.

\subsection{Parabolic subgroups}
Let $H=\mathrm{SO}(V)$ be the special orthogonal group of $V$ and let~$P$ be its parabolic subgroup stabilizing the isotropic line~$\ell$.
The Levi decomposition of~$P$ is of the form~$P=\NP\rtimes \MP$.
Here~$\NP$ is the \emph{unipotent radical}, which consists of Eichler transformations and acts as the identity on~$\ell$.
We consider it as a group of matrices acting on vectors of the form~\eqref{eq;columnbas} as
\begin{align*}
\NP=\left\{\nP(b)\coloneqq
\left(\begin{smallmatrix}
	1 & b^* & -q(b) \\
	0 & I_{V_0} & -b \\
	0 & 0 & 1
\end{smallmatrix}\right) : b\in V_0\right\},
\end{align*}
where~$b^*\in\Hom(V_0,F)$ is defined as~$b^*(x_0)=(b,x_0)$.
The group~$\MP\cong F^\times \times \mathrm{SO}(V_0)$ is the \emph{Levi factor} of~$P$, consisting of the isometries 
\begin{align*}
\mP(\alpha,h_0)\left(\begin{smallmatrix}
	x_2 \\
	x_0 \\
	x_1
	\end{smallmatrix}\right)
=
	\left(\begin{smallmatrix}
	\alpha x_2
	\\
	h_0(x_0)
	\\
	\alpha^{-1} x_1
	\end{smallmatrix}\right),\qquad
\alpha\in F^\times,\quad h_0\in\mathrm{SO}(V_0).
\end{align*}
When $h_0$ is trivial, we sometimes omit it from the notation. 

\subsection{Tube domains}\label{sec:tubedomains}
Let~$\domain=\domain_1\times\dots\times \domain_d$ be the Hermitian symmetric space associated to~$H$, where~$\domain_\ind$ denotes the Grassmannian of oriented negative definite planes in~$V_\ind$.
The Grassmannian~$\domain_\ind$, resp.~$\domain$, has a tube domain realization
\begin{align*}
\tubedom_\ind \coloneqq \{Z=X+iY\in V_0\otimes_{\sigma_\ind}\C : (Y,Y)<0\},
\end{align*}
resp.~$\tubedom \coloneqq \prod_{\ind=1}^d \tubedom_\ind$.
The identification of~$\tubedom_\ind$ with~$\domain_\ind$ is given by mapping~$Z$ to the plane spanned by $v_1$ and $v_2$ with orientation~$v_1\wedge v_2$, where~$v_1=X+\sigma_r(\genUU)+(q(Y)-q(X))\sigma_r(\genU)$ and~$v_2=Y-(X,Y)\sigma_r(\genU)$.
See~\cite[Part~II, Section~2.4]{1-2-3} for more details.

The action of the parabolic subgroup $P_{\ind}\coloneqq P \otimes_{\sigma_\ind}\RR$ on $\domain_\ind$ transfers to $\tubedom_\ind$ as follows.
The unipotent subgroup~$N_{\ind} \coloneqq \NP \otimes_{\sigma_\ind}\RR$ on $\domain_\ind$ acts by translations on~$\tubedom_\ind$, namely,
\begin{align*}
\nP(b) \colon \ \tubedom_\ind\to\tubedom_\ind,\quad Z\mapsto Z+b \quad \text{for all~$b\in V_0\otimes_{\sigma_\ind}\RR$.}
\end{align*}
The elements of the form~$\mP(\aa,h_0)\in
\MP \otimes_{\sigma_\ind}\RR$ on $\domain_\ind$ with~$\aa\in\RR^*$ and~$h_0\in \mathrm{SO}(V_0\otimes_{\sigma_\ind}\RR)$ act on~$\tubedom_\ind$ as
\begin{align*}
\mP(\aa,h_0) \colon \ \tubedom_\ind\to\tubedom_\ind, \quad Z\mapsto h_0(\aa Z).
\end{align*}

We fix a \emph{base point}~$Z^\ind_0\coloneqq i\sqrt{2}e^\ind_{n+1}$ on each~$\tubedom_\ind$, and denote by~$z^\ind_0$ the corresponding point in~$\domain_\ind$ under the identification explained above.
This is a negative-definite plane in~$V_\ind$ with orientation~$e_{n+1}^\ind\wedge e_{n+2}^\ind$.
Recall that the action of $P_\ind$ on $\tubedom_\ind$ is transitive. 
In fact, any~$Z=X+iY\in\tubedom_\ind$ can be written as~$h_Z(Z^\ind_0)$, where
\begin{equation}\label{eq;explchofh_Z}
    h_Z=\nP(X)\mP(\lvert q(Y)\rvert^{1/2},h_Y)\in P_{\ind}.
\end{equation}
Here~$h_Y\in\mathrm{SO}(V_0\otimes_{\sigma_\ind}\RR)$ is an isometry mapping~$\sqrt{2\lvert q(Y)\rvert}e^\ind_{n+1}$ to~$Y$.
The existence of such~$h_Y$ is ensured by the transitivity of the action of~$\mathrm{SO}(V_0\otimes_{\sigma_\ind}\RR)$ on the set of anisotropic vectors in~$V_0\otimes_{\sigma_\ind}\RR$ of the same norm.

\subsection{Weil representation}\label{sec;mixedmodelF}
We use the same the notations as in the introduction. 
For~$R=\A$, $\A_f$ or~$F_\infty$, we denote by~$G'_{R}$ the metaplectic cover of~$G(R)$.
Endow~$\A$ with the standard additive character~$\psi$ with Archimedean component~$\psi_\infty(x)=\ebf(x)=e(\tr(x))$ for~$x\in F_\infty$, where~$\tr(x)=\sum_j x_j$ and $e(z)\coloneqq\exp(2\pi i z)$.

Attached to the reductive dual pair~$(G,H)$ and the character~$\psi$, there is the Weil representation~$\omega$ \cite{weil-rep}.
In this section we recall its \emph{Schrödinger model} together with its \emph{mixed model} with respect to the isotropic line~$\ell$ of~$V$, in the same spirit of~\cite{LM80} and~\cite[Section~4.1]{Kudla-another}.
These models play a key role in computing the Fourier expansion of Kudla--Millson lifts of Hilbert cusp forms.


Let~$\big(W,\langle\cdot{,}\cdot\rangle\big)$ be the standard symplectic space over~$F$ of dimension~2.
This means that~$W=X\oplus Y$, with~$X= F e_X$ and~$Y=F e_Y$ such that
\begin{align*}
\langle e_X,e_Y\rangle = - \langle e_Y,e_X\rangle =1\quad\text{and}\quad \langle e_X,e_X\rangle=\langle e_Y,e_Y\rangle=0.
\end{align*}
We write vectors in~$W$ over the basis~$e_X,e_Y$ as \emph{row} vectors of coordinates, so that the action of the symplectic group $\Sp(W)$ on~$W$ translates to the right-action of~$\SL_2(F)$ on~$F^2$.

Recall that every symplectic space admits a (non-unique) \emph{polarization}, namely~$W=Z\oplus Z^\vee$ with symplectic form given by
\begin{align*}
(x\oplus\eta,y\oplus\xi)\mapsto \xi(x) - \eta(y).
\end{align*}
The choice of~$X$ and~$Y$ above induces a polarization for~$W$, since we may consider~$Y$ as the dual of~$X$.

Consider now the vector space~$V\otimes W$ endowed with the bilinear form~$(\cdot{,}\cdot)\otimes\langle\cdot{,}\cdot\rangle$.
Here the tensor products are over~$F$.
It is easy to check that $(\cdot{,}\cdot)\otimes\langle\cdot{,}\cdot\rangle$ is a symplectic form that makes $V\otimes W$ a symplectic space.

There are two natural polarizations of~$V\otimes W$.
The first one is given by
\begin{equation}\label{eq;firstpolF}
V\otimes W = \big(V\otimes X\big) \oplus\big( V\otimes Y\big),
\end{equation}
on which one realizes the \emph{Schrödinger model}~$\omega$ of the Weil representation.
In this model,~$\omega$ is a representation of~$G'_\A\times H(\A)$ on the space of Schwartz functions $\mathcal{S}(V\otimes X(\Ab))\cong\mathcal{S}(V(\Ab))$ arising from the first summand of~\eqref{eq;firstpolF}, see e.g.~\cite{Kudla-central} for an explicit description.


Another polarization is given by
\begin{equation}\label{eq;secondpolF}
V\otimes W = \Big( (V_0\otimes X) \oplus (\ell'\otimes W) \Big)
\oplus \Big( (V_0\otimes Y) \oplus (\ell\otimes W) \Big),
\end{equation}
on which one realizes the \emph{mixed model} of the Weil representation. 
By a slightly abuse of notation, we still denote it as $\omega$. 
It gives an action of~$G'_\A\times H(\A)$ on the space $\mathcal{S}\big(((V_0\otimes X)\oplus (\ell'\otimes W))(\Ab) \big)$ arising from the first summand of~\eqref{eq;secondpolF}, which can be recovered from the Schrödinger model as follows.

In analogy with Section~\ref{sec:decompofV}, we consider every Schwartz function~$\varphi\in\mathcal{S}(V(\Ab))$ as a function of~$(x_2,x_0,x_1)$, where $x_0\in V_0\otimes X(\Ab)\cong V_0(\Ab)$, $x_1\in \ell'\otimes X(\Ab)\cong \ell'(\Ab)\cong\Ab$ and~$x_2\in \ell\otimes X(\Ab)\cong \ell(\Ab)\cong\Ab$, with identifications given with respect to the choice of bases $e \in \ell$ and~$e'\in \ell'$ above.
Note that the pairing of~$\ell\otimes X$ with~$\ell'\otimes Y$ given by the restriction of $(\cdot{,}\cdot)\otimes\langle\cdot{,}\cdot\rangle$ becomes the product~$x_1x_2$ in~$\A$.

Let $(\eta_1,\eta_2)\in\Ab^2$. The partial Fourier transform
\begin{equation}\label{eq:passschömixmod}
\begin{split}
\mathcal{S}(V(\Ab)) & \longrightarrow \mathcal{S}(V_0(\Ab))\otimes \mathcal{S}((\ell'\otimes W)(\Ab)),\\
\varphi (x_2,x_0,x_1) & \longmapsto \Fc(\varphi)(x_0,\eta_1,\eta_2) \coloneqq
\int_{\Ab} \varphi(x_2,x_0,\eta_1)\psi(x_2\cdot \eta_2) \,dx_2
\end{split}
\end{equation}
is an isomorphism, and transfers the Weil representation from the Schrödinger model to the mixed model.
The additive Haar measure on $\A$ is self-dual with respect to Fourier transform.
In particular, $\vol (F\backslash \A)=1$ and $\vol(\widehat\Oc) = \dis^{-1/2}$.

More precisely,
\begin{equation}\label{mixed weil}
  \begin{split}
\Fc\big(\omega(g')\varphi\big)(x_0,\eta_1,\eta_2) &= \omega_0(g')\Fc(\varphi)(x_0,(\eta_1,\eta_2)g),\quad 
 g'\in G'_{\Ab};\\    
\Fc\big(\mP(\aa,h_0)\varphi\big)(x_0,\eta_1,\eta_2)&= |\aa|_{\A} \Fc(\varphi)(h_0^{-1}x_0,\aa\eta_1,\aa\eta_2),\quad
 \mP(\aa,h_0)\in \MP(\Ab);\\
\Fc\big(\nP(b)\varphi\big)(x_0,\eta_1,\eta_2) = \psi\Big(\big((b,x_0)&+q(b)\eta_1\big)\eta_2\Big)\Fc(\varphi)(x_0+b\eta_1,\eta_1,\eta_2),\quad 
 \nP(b)\in \NP(\Ab),
\end{split}
\end{equation}
where~$\omega_0$ is the Weil representation for~$V_0$ and~$g \in G(\Ab)$ is the image of~$g'$ under the canonical projection.
The passage from one model to the other enables us to rewrite theta functions via Poisson summation, as we now illustrate.
\begin{definition}
For any~$\varphi\in\mathcal{S}(V(\Ab))$,  the associated theta function is
\begin{align*}
\theta(g',h,\varphi) \coloneqq \sum_{x\in V(F)} \omega (g') \varphi(h^{-1} x),\quad g'\in G'_{\A},\ h\in H(\A).
\end{align*}
\end{definition}
Applying the partial Fourier transform, we may identify theta functions associated to the Schrödinger and the mixed models of the Weil representation as
\begin{equation}\label{theta PFT}
\sum_{x\in V(F)}\omega(g',h)\varphi(x) = \sum_{\substack{x_0\in V_0(F) \\ \eta\in F^2}} \mathcal{F}(\omega(g',h)\varphi)(x_0,\eta).
\end{equation}

\subsection{Lattices}\label{sec:lattices}
For~$a\in F$, let $a_\ind=\sigma_\ind(a)$ be the $\ind$-th component of $a$ under the embedding $F\subset F_\infty$.
For a subset $S\subset F$, let $S^+$ be the subset of totally positive elements of $S$, i.e.\
\[S^+=\{a\in S: \text{$a_\ind>0$ for all $\ind$}\},\]
so that~$\Oc^{\times, +}$ is the set of totally positive units of~$\Oc$. 
For any fractional ideal $\mathfrak{a}$, we define $\widehat{\mathfrak{a}} \coloneqq \mathfrak{a} \otimes \widehat{\Oc}$.

Let $L\subset V$ be an even integral $\Oc$-lattice, i.e.~a finitely generated~$\OO$-submodule of~$V$ such that~$L\otimes_{\OO} F=V$ and~$q(L)\subset \mathfrak{d}^{-1}$.
We denote by~$L^\vee$ the $\ZZ$-dual lattice of~$L$ with respect to~$q_\QQ\coloneqq\tr_{F/\QQ}q$, namely
\begin{align*}
L^\vee&\coloneqq \{ v\in L\otimes_\ZZ\QQ : B_\QQ(v,L)\subset\ZZ\}
= \{ v \in L \otimes_\Oc F : B(v,L) \subset \mathfrak{d}^{-1} \},
\end{align*}
where~$B$ and~$B_\QQ$ are the bilinear forms induced by~$q$ and~$q_\QQ$ respectively.
The discriminant group of~$L$ is the finite~$\Oc$-module $L^\vee/L$.
Note that for any $\Zb$-lattice $L \subset V$ satisfying $L \otimes \Qb = V$, we can find an order $\Oc' \subset \Oc$ such that $L$ is an $\Oc'$-module.
%
Let $\Sc_L$ be the subspace of~$\mathcal{S}(V(\Ab_{f}))$ generated by the characteristic functions~$\varphi_\mu$ of the cosets~$\mu + \widehat{L}$ for~$\mu\in \dual{L}/L
= \widehat{\dual{L}}/\widehat{L}$. 

\begin{definition}[{\cite[Section~4]{bruinier-reg}}]\label{def:Weilrepdisc}
Let~$\Gamma'$ be the preimage in~$G'_{F_\infty}$ of the Hilbert modular group~$\Gamma=\SL_2(\mathcal{O})$ and let~$\opcomf$ be the preimage of~$\widehat{\Gamma}\coloneqq \SL_2(\widehat{\mathcal{O}})$ in $G'_{\A_f}$.
We define the Weil representation~$\rho_L$ of~$\Gamma'$ on~$\mathcal{S}_L$ by
\begin{align*}
\rho_L(\gamma)\varphi_\mu=\overline{\omega}(\gamma_f)\varphi_\mu,\quad \gamma\in \Gamma^\prime,
\end{align*}
where~$\gamma_f$ is the unique element in~$\opcomf$ such that~$\gamma\gamma_f\in G(F)$, and where~$G(F)$ is identified with its image in~$G'_\A$ under the canonical splitting.
\end{definition}

The subspace~$\Sc_L$ is preserved under the action of~$\opcomf$, see~\cite[Section~3.2]{bruinier-reg}.
Note that~$\Sc_L$ is isomorphic to the group algebra~$\CC[\dual{L}/L]$.
We denote the standard basis vectors of the latter as~$\mathfrak{e}_\mu$.
For two functions $f=\sum_{\mu\in \dual{L}/L} f_\mu \ec_\mu$ and $g=\sum_{\mu\in \dual{L}/L} g_\mu \ec_\mu$ valued in~$\Sc_L$, we define the Hermitian pairing
\begin{align}\label{hermitian pairing}
\langle f, g \rangle \coloneqq\sum_{\mu\in \dual{L}/L} \overline{f_\mu} g_\mu.
\end{align}
Note that~$\langle \cdot{,} \cdot \rangle$ is antilinear in the first entry and linear in the second one.

\subsection{Hilbert modular forms}
We recall vector-valued Hilbert modular forms with respect to~$\rho_L$ following~\cite{bruinier-reg}.
By a slight abuse of notation, we fix a set $\{ \xi_i : \xi_i \in \Ab_f^\times , i= 1,\dots,h \}$ of representatives of the class group $\Cl(F)$ from \eqref{eq:Cl} and also denote the set of representatives as $\Cl(F)$, where $\xi_1=1$.


Let~$j(\gamma,\tau)=\prod_r(c_r\tau_r+d_r)$ for any~$\gamma=\big(\begin{smallmatrix}
    a & b \\ c & d
\end{smallmatrix}\big)\in G(F_\infty)\cong\SL_2(\RR)^d$.
We realize~$G'_{F_\infty}$ as the group of pairs~$(\gamma,\phi)$ with~$\gamma=\big(\begin{smallmatrix}
    a & b \\ c & d
\end{smallmatrix}\big)\in G(F_\infty)$ and~$\phi$ a holomorphic map on~$\HH^d$ such that~$\phi(\tau)^2=j(\gamma,\tau)$.
Note that $\phi(\tau)=\pm\sqrt{j(\gamma,\tau)}$, where~$\sqrt{\cdot}$ denotes the principal branch of the square root.
The group law of~$G'_{F_\infty}$ is
\[
\big(\gamma_1,\phi_1(\tau)\big)\big(\gamma_2,\phi_2(\tau)\big)
=
\big(\gamma_1\gamma_2,\phi_1(\gamma_2\tau)\phi_2(\tau)\big).
\]
For a lattice $L \subset V$ as in section \ref{sec:lattices}, 
and $k=(k_1,\dots,k_d)\in(\ZZ+\frac{n}{2})^d$.
We define a Petersson slash operator of weight~$k$ for~$\rho_L$ on functions~$f\colon\HH^d\to\mathcal{S}_L$ as
\begin{equation}\label{eq:petslashgen}
    \big(f|_{k,\rho_L}(\gamma,\phi)\big)(\tau)
=
(c_1\tau_1+d_1)^{-k_1+n/2}\cdots(c_d\tau_d+d_d)^{-k_d+n/2}\phi(\tau)^{-n} \rho_L(\gamma,\phi)^{-1} f(\gamma\tau)
\end{equation}
for all~$(\gamma,\phi)\in\Gamma'$.

\begin{definition}\label{def:vvhilbertasBr}
    A \emph{Hilbert modular form of weight~$k$ with respect to~$\rho_L$} is a holomorphic map~$f:\H^d\rightarrow \Sc_L$ such that
\begin{align*}
f|_{k,L}(\gamma,\phi) = f(\tau) \qquad \text{for all } (\gamma,\phi) \in \Gamma'.
\end{align*}
If~$F=\QQ$, we also assume that~$f$ is holomorphic at~$\infty$.
\end{definition}

Due to Koecher's principle, $f$ has the Fourier expansion
\begin{equation}\label{eq:Fexpdef}
f(\tau) = \sum_{\mu \in L^{\vee} / L} \sum_{\substack{\indcoef \in \{ 0 \} \cup F^{+}\\ \indcoef - q(\mu) \in\mathfrak{d}^{-1}}} 
\coef_\indcoef(f_\mu) \e(\indcoef \tau)\ef_\mu
\end{equation}
at $\infty$.
If~$F=\QQ$, the holomorphicity of~$f$ at~$\infty$ means that the Fourier expansion of~$f$ is as in~\eqref{eq:Fexpdef}.
We denote by $M_{k,L}$ (resp.~$S_{k,L}$) the space of Hilbert modular (resp.\ cusp) forms of weight~$k$ with respect to~$\rho_L$.
In the main body of the paper, we will be mainly concerned with the case of parallel weight~$k$ with~$k\equiv \frac{n}{2}$ mod~$\ZZ$.
In this case, the Petersson slash operator~\eqref{eq:petslashgen} boils down to
\[
\big(f|_{k,\rho_L}(\gamma,\phi)\big)(\tau)
=
\phi(\tau)^{-2k} \rho_L(\gamma,\phi)^{-1}
f(\gamma\tau).
\]

\subsection{Adelic Hilbert modular forms}\label{sec:groupvarexpl}
We now recall how to construct vector-valued adelic Hilbert modular forms on~$G(F)\backslash G'_\A$ from~$M_{k,L}$, where~$G(F)$ is embedded in~$G'_\A$ under the canonical splitting homomorphism.
The reader may refer to~\cite[Section~3.1]{G90}, \cite[Section~8]{HiragaIkeda} and~\cite[Section~1]{kudla_integrals} for further details.
For simplicity, we work with Hilbert modular forms of \emph{parallel weight}~$k$, since so is the weight of the Kudla--Millson theta function.

Any Hilbert modular form~$f\in M_{k,L}$ can be considered as an~$\mathcal{S}_L$-valued automorphic form on~$G(F)\backslash G'_\A$ as follows.
By the strong approximation theorem for~$G(\A)$, for every~$g'\in G'_\A$, there exists~$\gamma\in G(F)$, $g_\infty'=(g_\infty,\phi)\in G'_{F_\infty}$ and~$g_f'\in \opcomf$ such that~$g'=\gamma g_\infty' g_f'$.
We define the function $f^\# (g') = \sum\limits_{\mu \in L^\vee/L} f_\mu^\#(g') \ef_\mu$ by
\[
f_\mu^\#(g')\coloneqq \phi(\mathbf{i})^{-2k} \bar\omega(g_f') f_\mu(g_\infty \mathbf{i}), \quad g'\in G'_\A,
\]
where~$\mathbf{i}=(i,\dots,i)\in\HH^d$. We denote the space of such functions by
\begin{align}
M^\#_{k,L}\coloneqq\{f^\# : f\in M_{k,L}\}.
\end{align}
Conversely, for a given~$F\in M^\#_{k,L}$, we can construct~$f\in M_{k,L}$ such that~$F=f^\#$ as
\[f(\tau)=y^{-k/2}F(g_\tau'),
\qquad g'_\tau=(g_\tau,\phi_\tau),
\]
where~$\phi_\tau(\cdot)=\sqrt{j(g_\tau,\cdot)}=y^{-1/4}$ is a constant function, for~$\tau=x+iy\in\HH^d$.
Here we use the standard multi-index notation, so that~$y^{-k/2}=\prod_{r}y_r^{-k/2}$ and~$
g_\tau=(g_{\tau_1},\dots,g_{\tau_d})$, where~$g_{\tau_r}\coloneqq\big(\begin{smallmatrix}
    1 & x_r \\ 0 & 1
\end{smallmatrix}\big)\big(\begin{smallmatrix}
    {y_r}^{1/2} & 0 \\
    0 & {y_r}^{-1/2}
\end{smallmatrix}\big)$ and~$\tau_r=x_r+iy_r$.

To illustrate further properties of~$f^\#$, we need to introduce some notation.
For any ring~$R$, we define
\begin{align*}
    N(R)\coloneqq\left\{n(b)=\big(\begin{smallmatrix}
    1 & b \\
    0 & 1
\end{smallmatrix}\big) : b\in R\right\},\quad
M(R)\coloneqq \left\{
m(a)=\big(\begin{smallmatrix}
    a & 0 \\
    0 & a^{-1}
\end{smallmatrix}\big) : a\in R^\times
\right\}.
\end{align*}
Note that the connected component $M(F_\infty)^+$ of $M(F_\infty)$
containing the identity is isomorphic to~$\{m(a):a\in\RR^{+}\}^d$.
The metaplectic cover $G_R'$ splits when restricted to the preimages of $M(R)$ and $N(R)$.
We also denote by~$N(R)$ and~$M(R)^+$ their images in~$G'_{R}$ under the canonical splitting.

Let~$K'_\infty\subset G'_{F_\infty}$ be the metaplectic cover of~$\SO_2(\RR)^d$.
For any~$\theta=(\theta_1,\dots,\theta_d)\in(-\pi,\pi]^d$, we define
\[
u(\theta)=(u(\theta_r))_{r=1}^d\in\SO_2(\RR)^d,\quad u(\theta_r)=\big(\begin{smallmatrix}
    \cos\theta_r & \sin\theta_r \\
    -\sin\theta_r & \cos\theta_r
\end{smallmatrix}\big).
\]
Let~$\chi_{2k}$ denote the character
\begin{equation}\label{eq;charforF}
\chi_{2k}\colon K_\infty'\to\CC^\times, \qquad \big(u(\theta),\varphi(\tau)\big)\mapsto
\varphi(\mathbf{i})^{-2k}
\end{equation}
The proof of the following result is standard and hence omitted.

\begin{lemma}\label{lemma:basicpropfsharp}
  If~$f\in M_{k,L}$, then
\begin{align}
    &f^\#(\gamma g' u')= \bar\omega(u') f^\#(g') \qquad\text{for all $\gamma\in G(F)$, $g'\in G'_\A$ and $u'\in \opcomf$,}\label{lemma:basicpropfsharp:1}
        \\
    &f^\#(g'u'_\infty)=\chi_{2k}(u'_\infty)f^\#(g') \qquad\text{for all~$u'_\infty\in K_\infty'$ and~$g'\in G'_\A$}.
      \label{lemma:basicpropfsharp:2}
\end{align} 
\end{lemma}

For later use, we further define $\coef_\indcoef (f_{\mu},\xi)$ by
\begin{equation}
  \label{eq:FC-xi}
{\coef_\indcoef (f_{\mu},\xi)}
:=
\frac{e^{2\pi \tr (\indcoef a^{2})}}{a^{k}}
\int_{F \backslash \A} {f^\#_\mu(n(b) m(\xi) m(a))} \overline{\psi(b \indcoef)} \,db 
\cdot
\begin{cases}
  1 & k \in \Zb,\\
  \overline{\gamma(\xi, \psi)} &  k \in \Zb + \frac12\\
\end{cases}
\end{equation}
for $\indcoef\in F, a\in F_\infty^{+}$, where $\gamma(\cdot,\psi)$ is the Weil index \cite{Kudla-central}.
When $\xi$ is trivial,
this is the $\indcoef$-th Fourier coefficient ${\coef_\indcoef (f_{\mu})}$ of $f_\mu$ at infinity in \eqref{eq:Fexpdef}.

\subsection{The Kudla--Millson theta function}\label{sec:KMthetaintro}
We recall how to construct the Kudla--Millson theta function on~$\domain$ from~\cite{rosuyott} and \cite{kudla-cycles}.
We first construct a Kudla--Millson Schwartz function on each $\domain_\ind$. 
Let~$z^\ind_0\in\domain_\ind$ denotes the base point of~$\domain_\ind$ as in Section~\ref{sec:tubedomains}. 
Let $\mathfrak{h}=\mathfrak{p}+\mathfrak{k}$ be the Cartan decomposition of the Lie algebra of~$H_\ind\coloneqq H\otimes_{\sigma_\ind}\RR$.
The cotangent space of~$\domain_\ind$ at~$z^\ind_0$ is isomorphic to~$\mathfrak{p}^*$, the dual of $\mathfrak{p}$.
With respect to the pesudo-orthonormal basis $\{ e_j^\ind \}_{j=1}^{n+2}$, we have
\begin{equation}\label{eq;stdisopfrak}
\mathfrak{p}\cong\big\{\left(\begin{smallmatrix}
0 & X\\
X^t & 0
\end{smallmatrix}\right) : X\in \mathrm{Hom}(z^\ind_0,{z^\ind_0}^\perp)\cong M_{n,2}(\R) \big\}.
\end{equation}
The standard basis of~$\mathfrak{p}^*$ is given by the functionals~$\omega_{\alpha,\mu}$, whose value is the~$(\alpha,\mu)$-entry of the matrices in~$M_{n,2}(\RR)$ on the right-hand side of~\eqref{eq;stdisopfrak}.

Let~$K_\ind$ be the stabilizer of~$z^\ind_0$ with respect to the action of~$H_\ind$ on~$\domain_\ind$.
Let~$\mathcal{A}^2(\mathcal{D}_\ind)$ be the space of differential $2$-forms on $\mathcal{D}_\ind$.
The Kudla--Millson theta function on~$\domain_r$ arises from a special Schwartz function~$\varphi_{\mathrm{KM}}^{(\ind)}$ in~$\left[\mathcal{S}(V_\ind) \otimes \mathcal{A}^2(\mathcal{D}_\ind)\right]^{H_\ind}$.
The latter denotes the space of~$\mathcal{A}^2(\mathcal{D}_\ind)$-valued Schwartz functions~$\varphi$ on~$V_\ind$ that are invariant with respect to the action of~$H_\ind$, i.e.\ such that
\begin{equation}\label{eq;invariance}
   h^*\varphi(hx)=\varphi(x)\qquad\text{for all~$h\in H_\ind$ and~$x\in V_\ind$.} 
\end{equation}
The evaluation of differential forms on~$\domain_\ind$ at $z^\ind_0$ induces an isomorphism
\begin{align}
\left[\mathcal{S}(V_\ind) \otimes \mathcal{A}^2(\mathcal{D}_\ind)\right]^{H_\ind}
\cong\left[\mathcal{S}(V_\ind) \otimes {\bigwedge}^2\left(\mathfrak{p}^*\right)\right]^{K_\ind}.
\end{align}
Therefore, to define~$\varphi_{\mathrm{KM}}^{(\ind)}$ it is enough to construct it as an element in $\left[\mathcal{S}(V_\ind) \otimes \bigwedge^2\left(\mathfrak{p}^*\right)\right]^{K_\ind}$.

Recall that we write~$\genvec=\sum_j \baseco_je^\ind_j$.
The Kudla--Millson Schwartz function~$\varphiKM^{(\ind)}$ is the~$K_\ind$-invariant element of~$\Sc(V_\ind)\otimes\bigwedge^2(\mathfrak{p}^*)$ defined as
\begin{equation}\label{eq;KMschartzexplF=Q}
\varphiKM^{(\ind)}(\genvec,z_0)=\sum_{\indone,\indtwo=1}^n \big(\Gpol_{(\indone,\indtwo)}\varphi_0\big)(\genvec)
\omega_{\indone,1}\wedge\omega_{\indtwo,2},
\end{equation}
where
\begin{equation}\label{eq;stdgaus+Qab}
\varphi_0(\genvec)=\exp\Big(-\pi\sum_j \baseco_j^2\Big),
\qquad\text{and}\qquad
\Gpol_{(\indone,\indtwo)}(\genvec)=\begin{cases}
2 \baseco_\indone \baseco_\indtwo, & \text{if $\indone\neq\indtwo$,}\\
2 \baseco_\indone^2 - \frac{1}{2\pi}, & \text{otherwise.}
\end{cases}
\end{equation}

The spread of~$\varphiKM^{(\ind)}$ to the whole~$\domain_\ind$ produces an $H_\ind$-invariant element of $\Sc(V_\ind)\otimes\mathcal{A}^2(\domain_\ind)$.
It is \emph{closed} with respect to the standard differential of~$\mathcal{A}^2(\domain_\ind)$ and can be rewritten as
\begin{align}\label{eq;KMschartzexplF=Q2}
\varphiKM^{(\ind)}(\genvec,z)&=(h_\infty^{-1})^*\varphiKM^{(\ind)}(h_\infty^{-1}\genvec,z_0)
= \sum_{\indone,\indtwo=1}^n
\big(\Gpol_{(\indone,\indtwo)}\varphi_0\big)\big(h_\infty^{-1}\genvec\big)\cdot (h_\infty^{-1})^*\big(\omega_{\indone,1}\wedge\omega_{\indtwo,2}\big),
\end{align}
for every~$h_\infty\in H_\ind$ mapping~$z^\ind_0$ to~$z\in \domain_\ind$.

\begin{definition}\label{def;KMschwartzprod}
Let~$V_\infty=V_{\sigma_1}\times\dots\times V_{\sigma_d}$.
The \emph{Kudla--Millson Schwartz function}~$\varphiKM$ on~$\domain$ is defined as
\begin{equation}\label{eq:KMschwoverF}
\varphiKM (x,z) \coloneqq\bigwedge_{\ind=1}^d \varphiKM^{(\ind)} (x_\ind,z_\ind)\in \big[\mathcal{S}(V_\infty)\otimes\mathcal{A}^{2d}(\domain)\big]^{H(F_\infty)} ,
\end{equation}
where~$\varphiKM^{(\ind)}$ is the Kudla--Millson Schwartz function for~$V_{\ind}$. 
Moreover, when we use the tube domain realization $\tubedom$ of $\domain$, we will write the Kudla--Millson Schwartz function as $\varphiKM (x,Z)$.
\end{definition}

If it is clear from the context, we will sometimes drop the variable~$z$ or $Z$ from the notation.
Similarly as in \eqref{eq;KMschartzexplF=Q}, we have
\begin{align}
\varphiKM(\genvec,z_0) =
\sum_{\underline{\indone},\underline{\indone}'}
\big(\Gpol_{(\underline{\indone},\underline{\indone}')}\varphi_0\big)(\genvec) \omega_{\underline{\indone},\underline{\indone}'}, \quad
\genvec\in V_\infty,
\end{align}
where the sum is over tuples of the form~$\underline{\indone}^{\prime}=(\indone_1^{\prime},\dots,\indone_d^{\prime})$ with~$\alpha_\ind^{\prime} \in \{1,\dots,n\}$ and
\begin{equation}
  \label{eq:omega}
x\coloneqq(x_\ind)_\ind,\quad
\omega_{\underline{\indone},\underline{\indone}'} \coloneqq \wedge_\ind \lp \omega_{\indone_\ind,1}\wedge\omega_{\indone'_\ind,2} \rp,\quad
\Gpol_{(\underline{\indone},\underline{\indone}')}(\genvec) \coloneqq \prod_{\ind=1}^d \Gpol_{(\indone_\ind,\indone'_\ind)}(\genvec_\ind).
\end{equation}

We now introduce orthogonal Shimura varieties.
Let~$K\subset H(\Ab_{f})$ be a compact open subgroup that fixes~$\widehat{L}$ and acts as the identity on~$\widehat{L^\vee}/\widehat{L}$.
The complex points of the associated Shimura variety are
\begin{equation}\label{eq;SHimvar}
X_K\coloneqq H(F)\backslash \big(\domain\times H(\Ab_{f})/K\big).
\end{equation}
Recall that~$X_K$ splits into a finite union of connected components~$X_K=\coprod_j\Gamma_j\backslash\domain^+$, where $\domain^+=\prod_\ind\domain^+_\ind$ and~$\domain^+_\ind$ is one of the connected components of~$\domain_\ind$.
The arithmetic subgroups~$\Gamma_j$ are of the form~$H(F)^+\cap h_j K h_j^{-1}$ for some~$h_j\in H(\Ab_{f})$, where~$H(F)^+$ is the subgroup of~$H(F)$ preserving~$\domain^+$; see~\cite[Section~5]{kudla-remarks} for details.
We denote the connected components of~$X_K$ by~$Y_j=\Gamma_j\backslash\domain^+$.
When~$h_j=1$, we denote by~$Y_K=\Gamma_K\backslash\domain^+$ the component of~$X_K$ arising from~$\Gamma_K=H(F)^+\cap K$.

\begin{definition}
The \emph{scalar-valued Kudla--Millson theta function} associated with~$\varphi_f\in \mathcal{S}(V(\Ab_{f}))$ is
\begin{align*}
\theta(g',z,h,\varphiKM\varphi_f)
= \sum_{x\in V(F)}\omega(g')(\varphiKM(z)\varphi_f)(h^{-1}x)
\end{align*}
for~$g'\in G'_{\Ab}$, $z\in\domain$ and~$h\in H(\Ab_{f})$.
\end{definition}
In the variable~$(z,h)$, this theta function is a closed~$2d$-form on~$\domain\times H(\Ab_{F,f})/K$ invariant with respect to~$H(F)$, and therefore induces a closed form on~$X_K$.

\begin{definition}\label{def:vvKMtheta}
For any even integral $\Oc$-lattice $L\subset V$, we define the vector-valued Kudla--Millson theta function
\begin{align*}
\Theta_{\mathrm{KM}}^L(g',z,h)\coloneqq
\sum_{\mu\in L^\vee/L}\theta(g',z,h,\varphiKM\varphi_\mu)\mathfrak{e}_\mu,
\quad g'\in G'_\Ab,\ h\in H(\Ab_f).
\end{align*}
\end{definition}
If the lattice~$L$ is clear from the context, we drop it from the notation and simply write~$\Theta_{\mathrm{KM}}(g',z,h)$.
This theta function is a non-holomorphic $\mathcal{S}_L$-valued adelic Hilbert modular form of parallel weight $k=1+\frac{n}{2}$ in the variable $g'$.
Its cohomology class~$[\Theta_{\mathrm{KM}}^L]$ is \emph{holomorphic} in~$g'$, more precisely~$[\Theta_{\mathrm{KM}}^L]\in M^\#_{k,L} \otimes H^{2d}(X_K,\CC)$, and equals the generating series of the codimension~$d$ special cycles on~$X_K$.
See~\cite[Section~4]{rosuyott} for details.

\section{The Kudla--Millson lifts of Hilbert modular forms}\label{sec:KM}
In this section, we calculate the Fourier expansion of the Kudla--Millson lift of Hilbert modular forms of parallel weight $k = 1+n/2$
by means of the mixed model of the Weil representation, where the Poisson summation formula \eqref{theta PFT} allows us to unfold the theta integral.
Then, we use the expansion to prove an injectivity criterion on the Kudla--Millson lift.

Throughout this section, we assume that~$L$ splits off a hyperbolic plane over~$\mathcal{O}$, namely an~$\mathcal{O}$-lattice~
\begin{equation}
  \label{eq:Ua}
  U_{\mathfrak{a}}\coloneqq \mathfrak{a} \oplus \mathfrak{a}^{-1}\mathfrak{d}^{-1}
\end{equation}
for some fractional ideal~$\mathfrak{a}$, with quadratic form~$q_{U_{\mathfrak{a}}} (x,y) \coloneqq xy$.
Note that~$U$ is $\Zb$-unimodular and of signature~$(1,1)$ at every place of~$F$.
We define~$K_\infty^\prime$ to be inverse image of~$K_\infty$ in $G'_{F_\infty}$, 
and normalize the Haar measures on $G(F_\infty) \cong \Hb^d \times K_\infty$ to be 
$dg_\infty = \prod_{r = 1}^d \frac{dx_r dy_r}{ y_r^2} \frac{d\theta_r}2$ for $g_\infty = g_\tau u(\theta)$.
With respect to the self-dual measure $dx$ on $\Ab_f$, the group $\widehat\Gamma \subset G(\Ab_f)$ has volume $\frac1{\dis^{3/2}\zeta_F(2)}$.
It is a well-known result of Siegel \cite[section IV.1]{vandergeer} that
$$
\vol(\Gamma \backslash \SL_2(\R)^d) = \vol(K_\infty) \vol(\Gamma \backslash \Hb^d) = 2 \zeta_F(-1) (-2)^d\pi^{2d} =  2 \dis^{3/2} \zeta_F(2),
$$
which implies $\vol(G(F)\backslash G(\A)) =
\frac1{|G(F) \cap (\widehat\Gamma K_\infty)|}
\vol(\widehat\Gamma) \cdot \vol(\Gamma \backslash \Hb^d)\cdot \vol(K_\infty) =1$. 

\subsection{Adelic Kudla--Millson lifts}\label{sec;adelicKMlifts}
We write $L$ as an orthogonal direct sum
\begin{equation}\label{eq;splitLUL0}
L=U_\af\oplus L_0,
\end{equation}
where~$L_0=U_\af^\perp$,
and define $V_0 \coloneqq L_0 \otimes_{\Oc} F$.
We choose the generators $e$ and $e^\prime$ of the isotropic lines~$\ell$ and~$\ell'$, cf.\ Section~\ref{sec:decompofV}, to be the standard basis vectors of the hyperbolic~$F$-plane $U\otimes_{\Oc}F$.
We therefore have a decomposition \eqref{eq;decVF} of~$V$.

\begin{definition}
Let~$f\in S_{k,L}$ be a Hilbert cusp form of parallel weight~$k=\frac{n+2}{2}$ with respect to~$\rho_L$.
The \emph{Kudla--Millson lift of~$f$} is
\begin{align}\label{KM lift f}
  \KMliftL(f)(z,h) \coloneqq
  \int_{G(F) \backslash G(\A)} \langle f^\#(g'),\Theta_{\mathrm{KM}}^L(g',z,h) \rangle dg
\end{align}
for $z\in\domain, h\in H(\A_f)$. Here $g' \in G^\prime_{\A}$ is any preimage of $g \in G(\A)$.
We will omit $L$ from the notation if $L$ is clear in the context.
When we use the tube domain realization $\tubedom$ of $\domain$, we will write the Kudla--Millson lift of $f$ as $\KMliftL(f)(Z,h)$.
\end{definition}
Even though the cusp form $f$ and theta kernel are only functions on $G'_\A$ when $n$ is odd, their pairing in the integrand factors through $G(\A)$.
Note that due to the antilinearity of the pairing~$\langle\cdot{,}\cdot\rangle$ defined in~\eqref{hermitian pairing}, the lift
\[
\KMliftL\colon S_{k,L}\to\mathcal{A}^{2d}(X_K)
\]
is antilinear.

To use \eqref{theta PFT} to compute the Fourier expansion of the lift~$\KMliftL(f)$, we calculate the partial Fourier transform of the Schwartz functions~$\varphi=\varphi_{\mathrm{KM}}\varphi_\mu\in\mathcal{S}(V(\A))$ for $\mu\in L^\vee /L$.
Recall that we denote by~$\psi=\psi_\infty\psi_f$ the standard additive character of $\Ab$ over~$F$, with~$\psi_\infty(x)=e(\tr(x))$ for~$x\in F_\infty$.

We begin by computing the partial Fourier transform of~$\varphi_\mu$, i.e.,
\begin{align*}
\Fc_{f}(\varphi_\mu)(x_0,\eta_1,\eta_2) =
\int_{\Ab_{f}}\varphi_\mu(x_2,x_0,\eta_1)\psi_f(x_2\eta_2) \, dx_2.
\end{align*}
For any ideal $\mathfrak{a}$, denote by $\varphi_{\mathfrak{a}}$ the characteristic function of $\widehat\af \coloneqq  \mathfrak{a}\otimes \widehat{\Oc} \subset \A_f$.
Since~$L=L_0\oplus U_\af$ and~$U_\af$ is $\Zb$-unimodular, we have $L^\vee/L\cong L_0^\vee/L_0$. 
Therefore, for every coset~$\mu+L$ in~$L^\vee$ we may choose~$\mu=\mu_0\in L_0^\vee$.
This implies that $\varphi_\mu$ equals
\begin{equation}\label{eq;splitschwffpl}
\varphi_\mu(x)=\varphi_{\mu_0}(x_0)\varphi_{\mathfrak{a}}(x_1)\varphi_{(\mathfrak{ad})^{-1}}(x_2),\qquad
x=\left(\begin{smallmatrix}
 x_2 \\ x_0 \\ x_1
\end{smallmatrix}\right).
\end{equation}
and hence
\begin{equation}\label{eq;comparffin}
\begin{split}
\Fc_{f}&(\varphi_\mu)(x_0,\eta_1,\eta_2) 
= \varphi_{\mu_0}(x_0)\varphi_{\mathfrak{a}}(\eta_1)
\int_{\Ab_{f}} \varphi_{(\mathfrak{ad})^{-1}}(x_2) \psi_f(x_2\eta_2)\, dx_2 \\
&=\vol (\widehat{(\mathfrak{ad})^{-1}}) \varphi_{\mu_0}(x_0)\varphi_{\mathfrak{a}}(\eta_1) \varphi_{\mathfrak{a}}(\eta_2)
= \sqrt{\dis} \Nm(\af) \varphi_{\mu_0}(x_0)\varphi_{\mathfrak{a}}(\eta_1) \varphi_{\mathfrak{a}}(\eta_2).
\end{split}
\end{equation} 
We now proceed to compute the partial Fourier transform of~$\varphiKM(x,z)$ at the Archimedean places, namely,
\begin{align*}
\mathcal{F}_\infty(\varphi_\infty)(x_0,\eta_1,\eta_2) =
\int_{F_{\infty}} \varphi_\infty(x_2,x_0,\eta_1)\psi_\infty(x_2\eta_2)\, dx_2,\quad \varphi_\infty\in\mathcal{S}(V(F_\infty)).
\end{align*}

\begin{lemma}\label{lemma; FtransfKMschwartz}
Let~$z\in\domain$ and let~$h_\infty=\nP(b) \mP(\aa,h_0)\in P(F_\infty)$ be such that~$h_\infty(z_0)=z$. Then the partial Fourier transform  of~$\varphiKM(x,z)$ is
\begin{align*}
\mathcal{F}_\infty(\varphiKM(z))(x_0,\eta_1,\eta_2)
 & =\lvert \aa\rvert \cdot \ebf\Big(\big((b,x_0) + q(b)\eta_1\big)\eta_2\Big)
\varphi_0(h_0^{-1}(x_0+b\eta_1)) e^{-\pi a^2( \eta_1^2 + \eta_2^2)}
  \\
&\quad\times \sum_{\underline{\indone},\underline{\indtwo}}
  P_{(\underline{\indone},\underline{\indtwo})}(\aa \cdot (\eta_1 - i \eta_2))
  \widetilde{\Gpol}_{(\underline{\indone},\underline{\indtwo})} (h_0^{-1}(x_0+b\eta_1))
(h_\infty^{-1})^*\omega_{\underline{\indone},\underline{\indtwo}},
\end{align*}
where
$P_{(\underline{\indone},\underline{\indtwo})}(\eta) = \prod_{r=1}^d P_{(\indone_r,\indone'_r)}(\eta_r), 
\widetilde{\Gpol}_{(\underline{\indone},\underline{\indtwo})}(x_0) = \prod_{r=1}^d \widetilde{\Gpol}_{(\indone_r,\indone'_r)}((x_0)_r)$ with
\begin{equation}
  P_{(\indone,\indone')}(\eta)
\coloneqq
\begin{cases}
\eta^2
& \text{if~$\indone=\indone'=1$,}\\
\eta
&\text{if~$\indone=1 <\indone'$, }\\
  \eta
&\text{if~$\indone'=1 <\indone$, } \\
1
& \text{if~$\indone,\indone'>1$,}
\end{cases}
  \widetilde{\Gpol}_{(\indone,\indone')}(y)
\coloneqq
\begin{cases}
1
& \text{if~$\indone=\alpha'=1$,}\\
\sqrt{2}y_{\alpha'}
&\text{if~$\indone=1 <\alpha'$, } \\
\sqrt{2}y_{\indone} 
&\text{if~$\alpha'=1 <\indone$, } \\
  \Gpol_{(\indone,\alpha')}(y)
& \text{if~$\indone,\alpha'>1$,}
\end{cases}
\label{eq:tildeQ}
\end{equation}
and $y=\sum_{j=2}^{n+1} \baseco_je^r_j\in \sigma_r(V_{0})$. 

\end{lemma}
\begin{proof}
This follows from computing the partial Fourier transform of $\Gpol_{(\underline{\indone},\underline{\indtwo})}\varphi_0$, and we put the computation in Appendix \ref{sec:aphintproof}.
\end{proof}

To obtain the Fourier expansion of Kudla--Millson lift, we apply Poisson's summation~\eqref{theta PFT} to rewrite
\begin{align*}
\Theta_{\mathrm{KM}}(g',Z,h)=\sum_{\mu\in\dual{L}/L} \sum_{\substack{x_0\in V_0 \\ \eta\in F^2}} 
\omega_0(g')\Fc\big(\varphiKM(Z) \omega(h)\varphi_\mu\big)(x_0,\eta g)\mathfrak{e}_\mu,\quad h\in H(\A_f), Z\in\tubedom,
\end{align*}
where~$g\in G(\A)$ is the image of $g'\in G'_\A$ under the projection. 
A set of representatives for the action of~$G(F)$ by right-multiplication on~$F^2$ is given by
the zero tuple~$(0,0)$ with stabilizer~$G(F)$, and the tuple~$(0,1)$ with stabilizer $N(F)$.
Then we decompose~$\Theta_{\mathrm{KM}}=\Theta_{\mathrm{KM},0} + \Theta_{\mathrm{KM},1}$, where
\begin{align}\label{decomoposed theta functions}
\begin{split}
\Theta_{\mathrm{KM},0} (g',Z,h)  &\coloneqq   \sum_{\mu\in\dual{L}/L}\sum_{x_0 \in V_0} \omega_0(g') \Fc(\varphiKM(Z)\omega(h)\varphi_\mu)(x_0,(0, 0) )\mathfrak{e}_\mu,\\
\Theta_{\mathrm{KM},1} (g',Z,h) &\coloneqq \sum_{\mu\in\dual{L}/L}\sum_{x_0 \in V_0} \sum_{\gamma \in N(F) \backslash G(F)}
\omega_0(g') \Fc(\varphiKM(Z)\omega(h)\varphi_\mu)(x_0,(0, 1)\gamma g)\mathfrak{e}_\mu.
\end{split}
\end{align}
It is easy to see that $\Theta_{\mathrm{KM},i}$ is~$G(F)$-invariant and has the same weight as $\Theta_{\mathrm{KM}}$. 
Now the Kudla--Millson lift of~$f$ decomposes as $\KMliftL(f)(Z,h) = \KMliftL_{,0}(f)(Z,h) + \KMliftL_{,1}(f)(Z,h)$, where
\begin{align}\label{decomposed theta lift}
  \KMliftL_{,i}(f)(Z,h) &=
                          \int_{G(F) \backslash G(\Ab)} \langle f^\#(g') , \Theta_{\mathrm{KM},i}(g',Z,h) \rangle dg,\quad i=0,1.
\end{align}
By Lemma \ref{lemma; FtransfKMschwartz} and the formula for $h_Z$ in \eqref{eq;explchofh_Z}, $\Theta_{\mathrm{KM},0} (g',Z,h) = \Theta_{\mathrm{KM},0} (g',Y,h)$ and hence 
$$\KMliftL_{,0}(f)(Z,h) = \KMliftL_{,0}(f)(Y,h)$$ 
is independent on $X$.
Now we calculate $\KMliftL_{,i}(f)(Z,h)$ to get the Fourier expansion of the Kudla--Millson lift.

\begin{theorem}\label{Fourier expansion}
Let $f\in S_{k,L}$ be a Hilbert cusp form of parallel weight $k=\frac{n+2}{2}$. 
Suppose that~$L$ splits off~$U_\af
$ as in~\eqref{eq;splitLUL0}.
For every~$Z=X+iY\in\mathcal{H}$ 
and  $\zeta\in\A_f^\times$, the Fourier expansion of~$\KMliftL(f)(Z{,}\mP(\zeta))$ with respect to the cusp corresponding to the isotropic line~$\af \subset U_\af \subset V$ is
\begin{align} \label{KM expansion}
\begin{split}
  \KMliftL(f)(Z,\mP(\zeta)) &= \KMliftL_{,0}(f)(Y,\mP(\zeta))\\
  &\quad + 
  \sum_{\substack{ x \in V_0\\ q(x ) \gg 0 }}
C^L_\af(x , \zeta; f)
  \e\lp (X,x )\rp
  \sum_{\underline{\indone},\underline{\indtwo}}
  W_{(\underline{\indone},\underline{\indtwo})}   (Y, x )
      (h_Z^{-1})^*\omega_{\underline{\indone},\underline{\indtwo}},
\end{split}
\end{align}
where
\begin{equation}
  \label{eq:CL}
  \begin{split}
    C^L_\af(x , \zeta; f)
    &\coloneqq 
      \frac{(2\pi)^d \Nm(\mathfrak{a})^{}|\zeta|_\A}
      {\dis \zeta_F(2)}
      \sum_{(\delta, \xi)\in D(x , \zeta; \mathfrak{a})} |\xi|_\A^{\frac{n}{2}} \chi_{V_0} (\xi)
| \Nm(\indco)|^{n-1}
      \overline{\coef_{q(x /\indco)} (f_{ \xi x /\indco},\xi)}, \\
    D(x ,\zeta; \mathfrak{a}) \coloneqq & \{ (\delta,\xi) \in \Oc^\times \backslash F^\times \times \Cl(F) : \xi x /\delta\in \widehat{L_0^\vee}, \delta \in \xi \zeta^{-1} \widehat{\mathfrak{a}} \},
  \end{split}  
\end{equation}
and
\begin{equation}
  \label{eq:Wx0}
  \begin{split}
  W_{(\underline{\indone},\underline{\indtwo})}   (Y, x )
    \coloneqq 
|q(Y)|^{1/2}
&      \int_{(\R^{+})^d}
      \widetilde{\Gpol}_{(\underline{\indone},\underline{\indone}')}(h_Y^{-1} t x)
      P_{(\underline{\indone},\underline{\indone}')}(-i |q(Y)|^{1/2} t^{-1})
    \\
    &
    \quad
      \times t^{n-1} e^{-\pi \tr(4 q(x_{Y^\perp} ) t^2 + |q(Y)|/t^2)} \frac{dt}t.
  \end{split}
\end{equation}
The functions $ P_{(\underline{\indone},\underline{\indone}')}$ and $ \widetilde{\Gpol}_{(\underline{\indone},\underline{\indone}')}$ are defined in \eqref{eq:tildeQ}, and $q(x_{Y^\perp}) = q(x) + \frac{(x, Y)^2}{4|q(Y)|}$ is never trivial as $q(x)$ is totally positive. 
\end{theorem}

\begin{remark}
  \label{rmk:1}
     Theorem~\ref{Fourier expansion} provides the Fourier expansion of $\KMlift(f)(Z,\mP(\zeta))$ as a function of~$Z\in\tubedom$.
    This is sufficient for the applications below.
    More generally, one may compute the Fourier expansion of~$\KMlift(f)(Z,h)$, for~$h\in H(\A_f)$, by replacing each finite Schwartz function~$\varphi_\mu$ by~$\omega(h)\varphi_h$ and computing its partial Fourier transform in the mixed model. 
    For general~$h$, the resulting expression need not admit the simple lattice-theoretic factorization appearing in Theorem~\ref{Fourier expansion}.
    However, if~$h$ belongs to the parabolic~$P(\A_f)$, this partial Fourier transform can be computed explicitly using the mixed-model formulas in~\eqref{mixed weil}.
  \end{remark}
  \begin{remark}
      \label{rmk:2}
      The differential forms $    \omega_{\underline{\indone},\underline{\indtwo}}$ are linearly independent. When $\underline{\indone} = \underline{\indtwo} = (1, \dots, 1)$, the function $ \widetilde{\Gpol}_{(\underline{\indone},\underline{\indone}')}$ is trivial, and the function $
        W_{(\underline{\indone},\underline{\indtwo})} (Y, x)$
simplifies to
    \begin{align*}
        W_{(\underline{\indone},\underline{\indtwo})} (Y, x)
      &=
         4^{\frac{d(7-n)}4} \Nm (|q(Y)|)^{\frac{n+3}4}
         \Nm(q(x_{Y^\perp}))^{-\frac{n-3}4}
         \prod_{1 \le r \le d}
         K_{(n-3)/2}(4\pi \sqrt{q(x_{Y^\perp})_r |q(Y)_r|}),
    \end{align*}
by Equation (29) in section 4.5 of \cite{EMOT}.
Here $K_\nu$ is the $K$-Bessel function. 
    So the differential form $W(Y, x)$ is never identically 0. 
\end{remark}
\begin{proof}[Proof of Theorem~\ref{Fourier expansion}]
Since $f^\#$ is left $G(F)$-invariant, we unfold $\KMliftL_{,1}(f)(Z,h)$ as
\begin{align*}
  &
    \int_{G(F) \backslash G(\Ab)} 
    \sum_{\mu\in L^\vee/L}\sum_{x_0 \in V_0}
    \sum_{\gamma \in N(F) \backslash G(F)}
\overline{    f_\mu^\#(\gamma g')}
    \omega_0(\gamma g')
    \Fc(\varphiKM (Z)\omega(h)\varphi_\mu)(x_0,(0, 1)\gamma g) 
    dg \\
  =&
     \int_{N(F) \backslash G(\Ab)} 
    \sum_{\mu\in L^\vee/L} \sum_{x_0 \in V_0}
\overline{f_\mu^\#(g')}
   \omega_0(g')
   \Fc(\varphiKM (Z)\omega(h)\varphi_\mu)(x_0,(0, 1) g)
   dg.
\end{align*} 
Here $g' \in G^\prime_{\A}$ is any preimage of $g \in G(\A)$. 
Using the Iwasawa decomposition $G(\Ab) = N(\Ab) M(\Ab) \widehat{\Gamma} K_\infty$ and \eqref{mixed weil},  
the lift $\KMliftL_{,1}(f)(Z,m_P(\zeta))$ becomes 
\begin{equation*}
\begin{split}
  &
\vol(\widehat\Gamma)
    \sum_{\substack{\mu\in L^\vee/L\\x_0 \in V_0} }
    \int_{N(F) \backslash N(\A)}
 \int_{\A_f^\times/\hat\Oc^\times \times F_\infty^{\times, +}}  
\overline{f^\#_\mu(g_1 m(a) )} \\
  & \hspace{2in} \times    (\omega_0(g_1 m(a) ) \Fc(\varphiKM(Z)\varphi_\mu))(x_0,(0, \zeta a^{-1} ) )
\vol( K_\infty)
    d^\times a dg_1\\
=
  & 
    \vol(\widehat\Gamma)    |\zeta|_\A
    \sum_{\substack{\mu\in L^\vee/L\\x_0 \in V_0} }
    \sum_{\indco\in \Oc^\times \backslash F^\times} \sum_{\xi\in \Cl(F)} 
    \int_{F_\infty^{\times, +}}
    \int_{F \backslash \A}
    \overline{f^\#_\mu(n(b) m(\xi a_\infty) )} \psi(q(x_0)b) db
  \\
  &\hspace{2in} \times    (\omega_0(m(\xi a_\infty) ) \Fc(\varphiKM(Z)\varphi_\mu))(x_0,(0,\zeta \indco \xi^{-1} a_\infty^{-1}) )
(2\pi)^d    \frac{da_\infty}{a_\infty^3}.
\end{split}
\end{equation*}
where we have used \eqref{eq:Cl} to write $a = \indco^{-1} \xi a_\infty$ and $2^d\frac{da_\infty}{a_\infty^3} = \prod_{r = 1}^d \frac{dy_r}{y_r^2}$ for $a_\infty = (\sqrt{y_r})_{1 \le r \le d}$. 
Substituting Equation \eqref{eq;comparffin} and Equation \eqref{eq:FC-xi}, by the action of the Weil representation \cite[Equation (1.6)]{Kudla-central}, we obtain
\begin{align*}   
\begin{split}
= & (2\pi)^d\vol(\widehat\Gamma )    |\zeta|_\A
    \sum_{\substack{\mu\in L^\vee/L\\x_0 \in V_0} }
    \sum_{\indco\in \Oc^\times \backslash F^\times} \sum_{\xi\in \Cl(F)}
    |\xi|_\A^{\frac{n}{2}} \chi_{V_0} (\xi) 
    \int_{F_\infty^{\times, +}}
    a_\infty^{1 + n/2}e^{-2\pi \tr(q(x_0)a_\infty^2)} \overline{\coef_{q(x_0)} (f_{\mu},\xi)} \\
&\qquad \times \Fc_f(\varphi_\mu)(\xi x_0, (0,\zeta \indco \xi^{-1} ))
\times  a_\infty^{n/2} \Fc_\infty(\varphiKM(Z))(a_\infty x_0,(0,\indco  a_\infty^{-1}) )   \frac{da_\infty}{a_\infty^3}\\      
=& (2\pi)^d\sqrt{\dis} \Nm(\af) \vol(\widehat\Gamma)    |\zeta|_\A
    \sum_{\xi\in \Cl(F)} |\xi|_\A^{\frac{n}{2}} \chi_{V_0} (\xi) 
    \sum_{\mu_0 \in L_0^\vee/L_0}
    \sum_{\substack{x_0 \in V_0 \\ \xi x_0 \in \mu_0 +\widehat{L_0}}}
    \sum_{\substack{\indco\in \Oc^\times \backslash F^\times\\ \zeta\xi^{-1} \indco \in \widehat{\mathfrak{a}}}} \overline{\coef_{q(x_0)} (f_{\mu_0},\xi)}  \\
&\quad \times    \int_{F_\infty^{\times, +}} a_\infty^{n-1} e^{-2\pi\tr(q(x_0)a_\infty^2)}
    \Fc_\infty(\varphiKM(Z))(a_\infty x_0,(0,\indco  a_\infty^{-1}) )   \frac{da_\infty}{a_\infty}\\
=& (2\pi)^d\sqrt{\dis} \Nm(\af) \vol(\widehat\Gamma)    |\zeta|_\A
    \sum_{x \in V_0} \sum_{(\indco, \xi) \in D(x, \zeta; \af)} |\Nm(\indco)|^{n-1} |\xi|_\A^{\frac{n}{2}} \chi_{V_0} (\xi) 
    \overline{\coef_{q(x)/\indco^2} (f_{\xi x/\indco},\xi)} \\
&\quad \times  \int_{F_\infty^{\times, +}} t^{n-1} e^{-2\pi\tr(q(x)t^2)} \Fc_\infty(\varphiKM(Z))(tx,(0,t^{-1}) )   \frac{dt}{t}.
\end{split}
\end{align*} 
The last step above used the substitution $x \coloneqq  |\indco| x_0, t = a_\infty/|\indco|$, and $ \Fc_\infty(\varphiKM(Z))(x_0,0,\eta_2 ) $ is an even function by Lemma~\ref{lemma; FtransfKMschwartz}. 
Since $f$ is a cusp form, the sum over $x \in V_0$ only includes those vectors with totally positive norm. 
Applying Lemma~\ref{lemma; FtransfKMschwartz} with $h_\infty = h_Z$ in \eqref{eq;explchofh_Z} expresses the integral on the last line as
\begin{align*}
  |q(Y)|^{1/2}
  \sum_{\underline{\indone},\underline{\indtwo}} \ebf((X,x)) (h_Z^{-1})^*\omega_{\underline{\indone},\underline{\indtwo}}
& \int_{F_\infty^{\times, +}} t^{n-1} e^{-2\pi\tr(q(x)t^2)} \varphi_0(h_Y^{-1}(t x)) e^{-\pi |q(Y)|/t^2} \\
& \times
  P_{(\underline{\indone},\underline{\indtwo})}( - i|q(Y)|^{1/2}  /t)
  \widetilde{\Gpol}_{(\underline{\indone},\underline{\indtwo})}(h_Y^{-1}(t x))
  \frac{dt}{t}.
\end{align*}
It is straightforward to verify that
$$
e^{-2\pi\tr(q(x)t^2)}
\varphi_0(h_Y^{-1}(t \cdot  x))
=
e^{-4\pi\tr(q(x_{Y^\perp})t^2)}
$$
since $  \varphi_0(h_Y^{-1}  x) = e^{-2\pi(q(x_{Y^\perp}) - q(x_Y)) }$ by the definition of $h_Y$ in \eqref{eq;explchofh_Z} and $q(x) = q(x_Y) + q(x_{Y^\perp})$.
Putting these together finishes the proof.
\end{proof}

\subsection{Injectivity}
In this section we illustrate how to deduce from Theorem~\ref{Fourier expansion} an injectivity criterion for the Kudla--Millson lift~$\KMliftL$ for general totally real extensions~$F$ of~$\QQ$.
This generalizes the aforementioned injectivity results to the case of Hilbert modular forms.
We start with a lemma concerning the set $D(x_0, \zeta) \coloneqq  D(x_0,\zeta; \Oc)$ in defining the coefficient~$C^L_\af$ in \eqref{eq:CL}. 

\begin{lemma}\label{induction number of divisors}
  For any $(\delta_1,\xi_1) \in D(x_0,\zeta)$, the following map
  $$
  \iota(\delta_1, \xi_1): D(x_0/\delta_1, \xi_1) \to D(x_0, \zeta),~
  (\delta,\xi) \mapsto (\delta\delta_1,\xi)
  $$
  is well-defined and injective. 
  Furthermore, $(\delta_2, \xi_2) \in D(x_0,\zeta)$ is in the image of $\iota(\delta_1, \xi_1)$ if and only if
  $\delta_2/\xi_2 \in (\delta_1/ \xi_1) \widehat{\Oc}$. 
  In particular, if $\zeta x_0 \in \widehat{L_0^\vee}$, then $(1, \zeta) \in D(x_0, \zeta)$ and $\iota(\delta_1,\xi_1)$ is not surjective for any $(\delta_1,\xi_1) \neq (1, \zeta)$. 
\end{lemma}

\begin{proof}
  By definition,  any $(\delta, \xi) \in  D(x_0/\delta_1, \xi_1)$ satisfies
  \begin{equation}
    \label{eq:Dcond}
  \xi (x_0/\delta_1)/\delta \in \widehat{L_0^\vee},~
   \delta \in \xi \xi_1^{-1} \widehat{\Oc} ,
  \end{equation}
which imply $\xi x_0/(\delta_1\delta) \in \widehat{L_0^\vee}$ and 
$
\delta\delta_1 \in\delta_1 \xi \xi_1^{-1} \widehat{\Oc}
\subset  \xi \zeta^{-1} \widehat{\Oc}
$
since $\delta_1 \in \xi_1 \zeta^{-1} \widehat{\Oc} $. 
That means $\iota(\delta_1, \xi_1)(\delta, \xi) \in D(x_0, \zeta)$, and the map $\iota$ is well-defined. 
It is clearly injective.
Now $(\delta_2, \xi_2) \in D(x_0, \zeta)$ is in the image of $\iota(\delta_1, \xi_1)$ if and only if $(\delta_2/\delta_1, \xi_2) \in D(x_0/\delta_1, \xi_1)$.
The first condition in \eqref{eq:Dcond} is automatically satisfied, whereas the second is equivalent to $\delta_2/\delta_1 \in \xi_2 \xi_1^{-1} \widehat\Oc$, the same criterion in the lemma. 
Lastly if $\zeta x_0 \in \widehat{L_0^\vee}$, then $(1, \zeta) \in D(x_0, \zeta)$. 
Assume $\iota(\delta_1,\xi_1)$ is surjective. Then its image contains $(1, \zeta)$, which means $\xi_1 \in  \delta_1 \zeta \widehat{\Oc}$.
On the other hand, $ \delta_1 \zeta \in \xi_1 \widehat{\Oc}$ since $(\delta_1, \xi_1) \in  D(x_0, \zeta)$.
Together, they imply $\delta_1 \zeta \xi_1^{-1} \in  \widehat{\Oc}^\times$, which is impossible since 
 $(\delta_1,\xi_1) \neq (1, \zeta)$. 
\end{proof}

\begin{lemma}\label{one divisor}
Suppose $L_0 = U_{\Oc} \oplus M$, where $U_{\Oc} = \Oc u\oplus \mathfrak{d}^{-1} u'$ and $u,u'$ are the standard basis such that $(u,u')=1$.
Then for any $x = u + ru' + v, r\in\mathfrak{d}^{-1}, v \in M^\vee/ M$, we have $D(x,1)=\{(1,1)\}$. Note that $q(x)=r+q(v)$.
\end{lemma}
\begin{proof}
If~$(\delta,\xi)\in D(x,1)$, then $\delta\xi^{-1}\in\widehat{\Oc}$ and $\xi x/\delta\in\widehat{L_0^\vee}$, which implies
\begin{align*}
\frac{\xi u}{\delta} + \frac{\xi r u'}{\delta} + \frac{\xi v}{\delta} \in
\widehat{L_0^\vee} = \widehat{\Oc}u\oplus\widehat{\mathfrak{d}^{-1}}u'\oplus\widehat{M^\vee}.
\end{align*}
Hence $\xi\delta^{-1}\in\widehat{\Oc}$ and, moreover, $\xi\delta^{-1}\in\widehat{\Oc}^\times$.
By the choice of representatives in $\Cl(F)$, we obtain $\xi=1$ and $\delta\in\Oc^\times$.
This finishes the proof.
\end{proof}

\begin{theorem}\label{injectivity}
Let $L\cong L_0\oplus U_{\mathfrak{a}}$ be an even $\mathcal{O}$-lattice of signature $(n,2)$. 
Assume $q(\mu+ L_0) \supset (q(\mu) + \mathfrak{d}^{-1})\cap F^+$ for every $\mu \in L_0^\vee / L_0$, then $\KMliftL$ is injective.
\end{theorem}
\begin{proof}
Suppose $\KMliftL (f)(Z,h)\equiv 0$. Then the coefficient $C^L_\mathfrak{a}(x, \zeta; f)$ defined in \eqref{eq:CL} is 0 for all $\zeta \in \A_f$ and $x \in V_0$ by Remark \ref{rmk:2}. 
For $x\in V_0$ with $q(x) \gg 0$ and $\zeta \in\A_f^\times$ satisfying $\zeta x\in \widehat{L_0^\vee}$, we claim that $c_{q(x)}(f_{\zeta x}, \zeta) = 0$.
By taking $\zeta = 1$ and using the assumption, this implies that $f$ vanishes identically, and $\KMliftL$ is injective.

To prove the claim, we use induction on $n(x,\zeta) \coloneqq  |D(x, \zeta)|$.
If $n(x, \zeta) = 1$, then $c_{q(x)}(f_{\zeta x}, \zeta)$ and $C^L_\mathfrak{a}(x, \zeta; f)$ agree up to a non-zero multiplicative constant, and the claim holds.
In general, the vanishing of $C^L_\mathfrak{a}(x, \zeta; f)$ allows $c_{q(x)}(f_{\zeta x}, \zeta)$ to be written as a linear combination of $c_{q(x/\delta)}(f_{\xi  x/\delta}, \xi)$'s 
for $(\delta, \xi) \in D(x, \zeta)$ distinct from $(1, \zeta)$. 
By Lemma \ref{induction number of divisors}, we know that $n(x/\delta, \xi)$ is strictly smaller than $n(x, \zeta)$. By induction, $c_{q(x/\delta)}(f_{\xi  x/\delta}, \xi)$ vanishes for all such $(\delta, \xi)$, hence so does $c_{q(x)}(f_{\zeta x}, \zeta)$.   
\end{proof}


The assumption in Theorem \ref{injectivity} holds when $L_0$ splits off a hyperbolic plane. In that case, we give a simpler argument below. 

\begin{theorem}\label{injectivity two hyper}
Let $L\cong L_0\oplus U_{\mathfrak{a}}$ be an even $\mathcal{O}$-lattice of signature $(n,2)$. 
If~$L_0$ splits off a hyperbolic plane~$U_{\mathfrak{b}}$, then $\KMliftL$ is injective.
Furthermore, if $\mathfrak{a}=\mathfrak{b}=\Oc$, then the restriction of the lift to the connected component $Y_K$ of $X_K$ is injective as well.
\end{theorem}
\begin{proof}
Applying the action of $\mP(\eta_{\mathfrak{a}},h_{\mathfrak{b}})$ for $\eta_\mathfrak{a}\in \A_f^\times$ and $h_\mathfrak{b}\in \SO(V_0(\A_f))$, it suffices to consider the case $\mathfrak{a}=\mathfrak{b}=\Oc$. 
In this case, $q(\mu+ L_0) = q(\mu) + \mathfrak{d}^{-1}$ for any $\mu\in L_0^\vee/L_0$ and the injectivity follows from Theorem \ref{injectivity}.

To prove the restriction of $\KMliftL$ to the connected component $Y_K$ is injective when $\mathfrak{a}=\mathfrak{b}=\Oc$, we proceed as follows.
For any $\mu \in L_0^\vee/L_0$ and $r \in \mathfrak{d}^{-1}$, there exists $x\in \mu + L_0$ such that $q(x)=r + q(\mu)$ with $D(x,1)=\{(1,1)\}$ by Lemma \ref{one divisor}. Therefore, $c_{r+q(\mu)}(f_{\mu})$ is zero since it agrees with $C^L(x, 1; f)$ up to a non-zero constant. Therefore, $f=0$ and $\KMliftL$ is injective.
\end{proof}

\begin{remark}
After passing to the mixed model, the polynomial $P_{(1, 1)}(0, \gamma_2)\widetilde{\Gpol}_{(1,1)}(y_0)=-\gamma_2^2$
becomes negative when $\gamma_2\in\R^\times$ by Lemma \ref{lemma; FtransfKMschwartz}. For $F=\QQ$, this simplifies the proof in \cite{zuffetti_unfolding1}, where for each $\mu \in L_0 \otimes \mathbb{R}$ such that $q(\mu)>0$, one needs to find two different indexes $\alpha, \alpha' \in\{1, \ldots, n\}$, 
and $h \in H(\RR)$, such that~$\mathcal{P}_{(\alpha, \alpha'), h^{\#}, 1,0}(h^{\#}\mu)>0$, see Section~\ref{sec:comparison} for details on the polynomial~$\mathcal{P}_{(\alpha, \alpha'), h^{\#}, 1,0}$.
\end{remark}

\subsection{Comparison of the Fourier expansion over $\Q$ using Borcherds' approach}\label{sec:comparison}
We explain the relation between the Fourier expansion of the Kudla–Millson lift over~$F=\QQ$, obtained by passing to a mixed model of the Weil representation as described in Section~\ref{sec;adelicKMlifts}, and the expansion computed in~\cite{zuffetti_unfolding1} using Borcherds’ method~\cite{borcherds_singularities}.
Borcherds’ unfolding is based on decomposing certain Siegel theta functions -- attached to~$L$ and to (not necessarily harmonic) homogeneous polynomials -- with respect to the splitting~$L=U\oplus L_0$.

This section can be viewed as a dictionary that translates such decomposed theta functions into the language of the mixed model of the Weil representation and the tube domain model~$\tubedom$.
\\


Let~$F=\QQ$ and let
\[
\Delta\coloneqq\sum_{j=1}^{n+2}\frac{\partial^2}{\partial y_j^2}
\]
be the standard Laplacian on~$\RR^{n+2}$, where~$y_j$ is the coordinate of $L\otimes\RR$ with respect to the pseudo-orthonormal basis vector~$e_j\coloneqq e_j^1$ as in Section~\ref{sec:prelim}.

For any negative definite plane~$z\in\mathcal{D}=\mathcal{D}_1$, let $\genU_z$ and $\genU_{z^\perp}$ be the orthogonal projection of $\genU$ in $z$ and $z^\perp$ respectively.
Following Borcherds~\cite[Section~5]{borcherds_singularities}, we denote by~$w$ (resp.~$w^\perp$) the orthogonal complement of~$\genU_z$ (resp.~$\genU_{z^\perp}$) in~$z$ (resp.~$z^\perp$).
Furthermore, for every~$h\in H(\RR)$ we define
\begin{align}\label{h inverse sharp}
(h^{-1})^\#\colon L\otimes\RR\to L\otimes\RR, \qquad v\mapsto h^{-1}(v_{w^\perp}+v_w),
\end{align}
where~$z=hz_0$.

Let~$\mathcal{P}$ be a homogeneous polynomial of degree~$(m^+,m^-)$ on~$V(\RR)$, namely, homogeneous of degree~$m^+$ in the coordinates~$y_1,\dots,y_n$, and homogeneous of degree~$m^-$ in~$y_{n+1}$, $y_{n+2}$.
In~\cite[p.~508]{borcherds_singularities} Borcherds decomposes~$\mathcal{P}$ in terms of homogeneous polynomials $\mathcal{P}_{(h^{-1})^\#,r^+,r^-}$ on $(h^{-1})^\#(V(\RR))$ 
of degree respectively $(m^+-r^+,m^--r^-)$, as
\begin{equation}\label{eq;decomppolasinbor}
\mathcal{P}\big(h^{-1}v\big)
=
\sum_{r^+,r^-}(v,\genU_{z^\perp})^{r^+} (v,\genU_z)^{r^-} \mathcal{P}_{(h^{-1})^\#,r^+,r^-}\big((h^{-1})^\#v\big).
\end{equation}
For the Fourier expansion of the Kudla--Millson lift over~$F=\QQ$, it is enough to consider the degree $(2,0)$ homogeneous polynomials
\[
\mathcal{P}_{(\indone,\indtwo)}(v)=2y_\indone y_\indtwo\qquad \text{with }\indone,\indtwo=1,\dots,n.
\]

We are now ready to recall the Fourier expansion~\cite[(5.18)]{zuffetti_unfolding1} of~$\Lambda_{\mathrm{KM}}$ over the rationals.
For any $x\in V$, we denote $x^2 = (x,x) = 2 q(x)$.

\begin{theorem}[\cite{zuffetti_unfolding1}]\label{thm:fromzufunf}
    Let~$F=\QQ$, and let~$x_0\in \mu + L_0$, with~$\mu\in L_0^\vee/L_0$ and~$\qform(x_0)>0$.
    The Fourier coefficient of index~$x_0$ of the Kudla--Millson lift of a weight $k=1+n/2$ elliptic cusp form~$f$ is
    \begin{equation}\label{eq;coefffromunfinj1}
    \begin{split}
        &2\lvert q(Y)\rvert^{1/2}
    \sum_{\substack{\indco\in\ZZ_{\geq 1} \\ \indco|x_0}}
    \coef_{q(x_0)/\indco^2}(f_\mu)
    \sum_{\alpha,\alpha^{\prime}=1}^n\int_0^{+\infty} 
    \exp\Big(-\frac{2\pi y x_0^2}{\indco^2}+\frac{2\pi y(x_0,Y)^2}{\indco^2Y^2} + \frac{\pi \indco^2 Y^2}{2y}\Big)
     \\
    &\times y^{(n-1)/2}\sum_{r=0}^2\Big(\frac{\indco}{2iy}\Big)^{r}
   \exp\Big(-\frac{\Delta}{8\pi y}\Big)\Big(\mathcal{P}_{(\alpha,\alpha^{\prime}),(h^{-1})^\#,r,0}\Big)((h^{-1})^\#x_0/\indco)\,dy
   \cdot(h^{-1})^*\omega_{\alpha,\alpha^{\prime}}    \end{split}
\end{equation} 
Here~$h$ denotes any isometry in~$H(\RR)$ mapping the base point~$Z_0$ to~$Z\in\tubedom$.
\end{theorem}

The coefficient~\eqref{eq;coefffromunfinj1} is a complex-valued function on~$\tubedom$ that depends on the imaginary part~$Y$ of~$Z\in\tubedom$.
In fact, one can show that~$\exp\big(-\frac{\Delta}{8\pi y}\big)\big(\mathcal{P}_{(\alpha,\alpha^{\prime}),(h^{-1})^\#,r,0}\big)((h^{-1})^\#x_0/\indco)$ depends only on~$Y$ even if the choice of~$h$ depends on~$Z$; see~\cite[Lemma~5.6]{zuffetti_unfolding1}.

Let~$P(\RR)$ be the parabolic subgroup of~$H(\RR)$ stabilizing the isotropic line~$\ell=\RR\genU$.
In what follows, we show that if we choose~$h$ to be the isometry~$h_Z\in P(\RR)$ as in~\eqref{eq;explchofh_Z}, then~\eqref{eq;coefffromunfinj1} boils down to a special case of Theorem~\ref{Fourier expansion}.
This serves a twofold purpose: It provides a consistency check for Theorem~\ref{Fourier expansion} and explains how to pass from the aforementioned Borcherds' formalism to the (more standard) mixed model of the Weil representation.

The following result may be considered as a dictionary to translate Borcherds' formalism in terms of explicit functions on~$\tubedom$.
Its proof is postponed in Appendix~\ref{sec;appborc}.

\begin{lemma}\label{lemma:lemtbcfromunf1}
    Let~$Z=X+iY\in\tubedom$ and let
    \[
    h_Z=\nP(X)\mP(\lvert q(Y)\rvert^{1/2},h_Y)\in P(\RR)
    \]
    be an isometry mapping the base point $Z_0$ to~$Z$.
    In particular,~$h_Y\in\bigO(V_0\otimes \RR)$ is any isometry mapping~$\sqrt{2\lvert q(Y)\rvert}e_{n+1}$ to~$Y$.
    \begin{enumerate}[label=(\roman*)]
    \item $\varphi_0(h_Y^{-1}x_0) = \exp(-\pi x_0^2 + 2\pi(x_0,Y)^2/Y^2)$ for all~$x_0\in V_0(\RR)$. 
    \label{item;lemtbcfromunf10}
        \item $(h^{-1}_Z)^\#v=h^{-1}_Y(v_0-(v,\genU)X)$ for all~$v\in V(\RR)$, where~$v_0$ denotes the orthogonal projection of~$v$ to~$V_0(\RR)$.
        In particular, $(h^{-1}_Z)^\#v_0=h^{-1}_Yv_0$.
        \label{item;lemtbcfromunf11}
        \item If~$x_0\in V_0(\RR)$, then
        \label{item;lemtbcfromunf12}
        \begin{align*}
            &\mathcal{P}_{(\alpha,\alpha^{\prime}),(h^{-1}_Z)^\#,r,0}((h^{-1}_Z)^\#x_0)
            \\
            &
            \quad
            =
            \begin{cases}
                -4Y^2(\genU,e_\alpha)(\genU,e_{\alpha^{\prime}}) & \text{if $r=2$},
                \\
                2\sqrt{2}\lvert Y\rvert\Big((\genU,e_\alpha)(h^{-1}_Yx_0,e_{\alpha^{\prime}}) + (\genU,e_{\alpha^{\prime}})(h^{-1}_Yx_0,e_\alpha)\Big) & \text{if $r=1$},
                \\
                2(h^{-1}_Yx_0,e_\alpha)(h^{-1}_Yx_0,e_{\alpha^{\prime}}) & \text{if $r=0$.}
            \end{cases}
        \end{align*} 
        \item \label{item;lemtbcfromunf13}
        If~$x_0\in V_0(\RR)$, then
        \begin{align*}
            &\exp\Big(-\frac{\Delta}{8\pi y}\Big)\Big(\mathcal{P}_{(\alpha,\alpha^{\prime}),(h^{-1}_Z)^\#,r,0}\Big)((h^{-1}_Z)^\#x_0)
            \\
            &
            =\begin{cases}
                \mathcal{P}_{(\alpha,\alpha),(h^{-1}_Z)^\#,0,0}(h^{-1}_Yx_0) - \frac{1}{2\pi y}& \text{if $r=0$ and $\alpha=\alpha^{\prime}>1$,}
                \\
                \mathcal{P}_{(\alpha,\alpha^{\prime}),(h^{-1}_Z)^\#,r,0}(h^{-1}_Yx_0) & \text{otherwise.}
            \end{cases}
        \end{align*}
    \end{enumerate}
\end{lemma}

We now compare the Fourier coefficients à la Borcherds computed in Theorem~\ref{thm:fromzufunf} with the ones in Theorem~\ref{Fourier expansion}.
\begin{proposition}\label{comparison adelic classical}
    Let~$F=\QQ$.
    The Fourier expansion in Theorem~\ref{Fourier expansion} equals the Fourier expansion~\eqref{eq;coefffromunfinj1} computed in~\cite{zuffetti_unfolding1} using Borcherds' unfolding method.
\end{proposition}
\begin{proof}
It is easy to use Lemma~\ref{lemma:lemtbcfromunf1} to show that the exponential part contained in the integral appearing in Theorem~\ref{Fourier expansion} is
\begin{equation}\label{eq;inprfdurshri}
\begin{split}
&\exp (-2\pi \qform(x_0/\indco) a^2)\widetilde{\varphi_0}\big(h_Y^{-1}ax_0/\indco,0,\lvert \qform(Y)\rvert^{1/2}\indco a^{-1}\big)
\\
    &=
    \exp\Big(
    -\frac{2\pi a^2x_0^2}{\indco^2} + \frac{\pi \indco^2 Y^2}{2a^2} + \frac{2\pi a^2 (x_0,Y)^2}{\indco^2 Y^2}
    \Big).
    \end{split}\end{equation} 

We now relate Borcherds' decomposed polynomials with the polynomials~$\widetilde{Q}_{\indone,\indtwo}$ appearing in Theorem~\ref{Fourier expansion}.
Without loss of generality, we may assume that~$\indone\leq\indtwo$.
Lemma~\ref{lemma:lemtbcfromunf1} \ref{item;lemtbcfromunf12} and~\ref{item;lemtbcfromunf13} imply that
\begin{equation}\label{eq;borchpolinQ}
\begin{split}
&\exp\Big(-\frac{\Delta}{8\pi a^2}\Big)(\mathcal{P}_{(\indone,\indone'),(h_Z^{-1})^\#,r,0})((h_Z^{-1})^\# x_0/\indco) \\
&=
\begin{cases}
a^{-2}
\widetilde{\Gpol}_{(\indone,\indtwo)}(h_Y^{-1} a x_0/\indco) P_{(\indone,\indone')}(-i |q(Y)|^{1/2} \delta a^{-1})
& \text{if $\alpha,\alpha'>1$ and~$r=0$,}\\
2i\indco^{-1} 
\widetilde{\Gpol}_{(1,\indtwo)}(h_Y^{-1} a x_0/\indco) P_{(1,\indone')}(-i |q(Y)|^{1/2} \delta a^{-1})
& \text{if $\indone=1$, $\indtwo>1$ and $r=1$,} \\
-\frac{4a^2}{\indco^2} 
\widetilde{\Gpol}_{(1,1)}(h_Y^{-1} a x_0/\indco) P_{(1,1)}(-i |q(Y)|^{1/2} \delta a^{-1})
& \text{if $\alpha=\alpha'=1$ and $r=2$,}\\
0 & \text{otherwise.}
\end{cases}
        \end{split}
    \end{equation}    

    The proof ends with a direct check, first performing the change of variable~$a=\sqrt{y}$ in the integral and then substituting the exponential part~\eqref{eq;inprfdurshri} and the polynomials~\eqref{eq;borchpolinQ} into the Fourier coefficients~\eqref{eq;coefffromunfinj1}.

\end{proof}

\section{Harmonicity of the lifts}\label{sec:harmonic}

Let~$\mathfrak{g}$ and~$\mathfrak{h}$ be the Lie algebras of~$G'_{F_\infty}$ and~$H(F_\infty)$.
Note that since~$G'_{F_\infty}$ is a double cover of~$\SL_2(\RR)^d$, and hence locally isomorphic to~$\SL_2(\RR)^d$, then~$\mathfrak{g}\cong\mathfrak{sl}_2(\RR)^d$.
We denote by~$d\omega$ the \emph{differential} of the Weil representation.
This is a representation of~$\mathfrak{g}\times\mathfrak{h}$ on~$\mathcal{S}(V_\infty)$, which extends to the product~$U(\mathfrak{g})\times U(\mathfrak{h})$ of universal enveloping algebras.

Recall that~$U(\mathfrak{g})\cong U(\mathfrak{sl}_2(\RR))^{\otimes d}$. The same split holds for the center of~$U(\mathfrak{g})$.
Let~$C_{G,r}$ be the Casimir element of the~$r$-th copy of~$\mathfrak{sl}_2(\RR)$ appearing in~$\mathfrak{g}$ and 
$$C_G \coloneqq  \sum_{i=1}^d C_{G,r}.$$
A similar notation is used when~$G$ is replaced with~$H$.

The following result follows from the case of~$F=\QQ$ considered by Shintani~\cite[Lemma~1.5]{shintani} and Howe~\cite{Howe}, see also~\cite[Theorem~II.4.1 and Lemma~II.4.1]{millson}.
\begin{theorem}[Shintani\footnote{Shintani works with the element~$C_Q$ defined in~\cite[(1.18)]{shintani} in place of the Casimir element~$C_G$.
    The latter is denoted by~$\Delta$ in~\cite[p.~153]{bump}.
    Comparing the two, we see that~$C_Q=-4 C_G$.}, Howe]\label{thm;howe}
    Let~$\mathfrak{z}$ and~$\mathfrak{z}'$ be the centers of~$U(\mathfrak{h})$ and~$U(\mathfrak{g})$.
    Then~$d\omega(\mathfrak{z})$ and~$d\omega(\mathfrak{z}')$ coincide as algebras of operators on~$\mathcal{S}(V_\infty)$.
    Furthermore, the Casimir elements satisfy
    \[
    d\omega(C_H)=d\omega(-4C_G)+\conmu,\qquad\text{where $\conmu\coloneqq d\Big(\frac{4-n^2}{4}\Big)$.}
    \]
\end{theorem}

Let~$\varphi\in\mathcal{S}(V_\infty)\otimes\mathcal{A}^{2d}(\domain)$,
and let
\[
\theta(g'_\infty,\varphi)\coloneqq\sum_{\mu\in \dual{L}/L}\sum_{v\in \mu + L}\big(\omega(g'_\infty)\varphi\big)(v)\mathfrak{e}_\mu,
\qquad g_\infty' \in G_{F_\infty}',
\]
be the $\CC[L^\vee/L]$-valued theta function corresponding to~$\varphi$.
Note that here we drop the variable~$h\in H(\A_f)$ appearing in Definition~\ref{def:vvKMtheta}, since the harmonicity of the lift does not depend on the connected component of~$X_K$ to which the differential form is restricted.
In Appendix~\ref{sec:append:diff} we recall how to differentiate such theta functions under the Casimir operators.

From now on, we assume that~$\varphi\in[\mathcal{S}(V_\infty)\otimes\mathcal{A}^{2d}(\domain)]^{H(F_\infty)}$ has parallel weight $k\in\ZZ+\frac{n}{2}$. 
Let~$g_\tau'=(g_\tau,y^{-1/4})\in G'_{F_\infty}$ be the standard element mapping~$\mathbf{i}$ to~$\tau\in\HH^d$.
Then the new theta function
\begin{equation}\label{eq;newthetatau}
\theta(\tau,\varphi)\coloneqq y^{-k/2}\theta(g_\tau',\varphi)
\end{equation}
is a~$\CC[L^\vee/L]\otimes\mathcal{A}^{2d}(\domain)$-valued smooth map on~$\HH^d$ that behaves as a Hilbert modular form of parallel weight~$k$ with respect to the Weil representation~$\rho_L$, see Section~\ref{sec:groupvarexpl}.

\begin{definition}
Let~$\Lambda_\varphi\colon S_{k,L}\to\mathcal{A}^{2d}(\domain)$ be the theta lift defined as
\begin{align*}
\Lambda_\varphi(f)\coloneqq \big(\theta(\cdot{,}\varphi),f\big)_{\Pet}
= \int_{\Gamma\backslash \HH^d} y^{k} \langle \theta(\tau,\varphi),f(\tau)\rangle \,d\mu,
\end{align*}
where $d\mu$ is the standard invariant measure of~$\HH^d$.
\end{definition}

Let~$\Omega\coloneqq\sum_r\Omega_r$, where~$\Omega_r$ is the invariant Laplace operator on the $r$-th grassmannian~$\domain_r\hookrightarrow \mathcal{D}$, see~\cite[Section~4.1]{bruinier-habil} for details.
The differential form~$\Lambda_\varphi(f)$ is an eigenfunction of~$\Omega$, as illustrated in the following result.

\begin{theorem}\label{thm:omegatodeltak}
    Let~$\varphi\in[\mathcal{S}(V_\infty)\otimes\mathcal{A}^{2d}(\domain)]^{H(F_\infty)}$ be of parallel weight~$k$.
    Then
    \[
    \Omega\Lambda_\varphi(f)=\Big(dk(2-k) -\xi \Big)\Lambda_\varphi(f).
    \]
    In particular, if~$k=1+n/2$, then~$\Omega\Lambda_\varphi(f)=0$, namely~$\Lambda_\varphi(f)$ is a \emph{harmonic} differential form.
\end{theorem}

\begin{proof}
Similarly as in~\cite[Section~2.1]{bump}, let~$\mathcal{C}^\infty_L(\HH^d,k)$ be the space of smooth functions $F\colon\HH^d\to\CC[L^\vee/L]$ such that
\begin{equation}\label{eq;unitmodtransf}
F(\gamma\tau)=\Big(\frac{\phi(\tau)}{\lvert \phi(\tau)\rvert}\Big)^{2k} \rho_L(\gamma,\phi)F(\tau)\qquad\text{for all $(\gamma,\phi)\in\Gamma'$}
\end{equation}
and such that they are square-integrable with respect to the Petersson inner product.
The latter is defined as usual: If $F_1,F_2$ behave as in~\eqref{eq;unitmodtransf}, then~$\langle F_1,F_2\rangle\colon\HH^d\to\CC$ is~$\Gamma$-invariant, and we denote
\[
((F_1{,}F_2))_{\Pet}\coloneqq\int_{\SL_2(\Oc)\backslash\HH^d} \langle F_1(\tau),F_2(\tau)\rangle\,d\mu
\]
whenever convergent.
Examples of functions in~$\mathcal{C}^\infty_L(\HH^d,k)$ are the maps~$y^{k/2}f$ where~$f\in S_{k,L}$.

By Corollary~\ref{cor:omtodelfortheta} we deduce that
    \begin{align*}
        \Omega\Lambda_\varphi(f)
        &=
        \Big( \Big(\Omega
        y^{k/2}\theta(\cdot,\varphi),
        y^{k/2}f\Big) \Big)_{\Pet}
        =
        4 \Big(\Big(\Delta_k\big(
        y^{k/2}
        \theta(\cdot{,}\varphi)\big),
        y^{k/2}
        f\Big) \Big)_{\Pet}
        -\xi\Lambda_\varphi(f).
    \end{align*}
    By \cite[Satz~3.1]{roelckeI}, see~\cite[Section~3.7]{völz} for a generalization to the vector-valued case and~\cite[Section~2.2]{BabaCP} for the Hilbert modular form case, the Laplacian~$\Delta_k$ is a symmetric operator on~$\mathcal{C}^\infty_L(\HH^d,k)$.
    Since~$\theta(\cdot{,}\varphi)$ minus a suitable Eisenstein series is of rapid decay with respect to~$y$ diverging to any cusp, see e.g.~\cite[Section~6.2]{bruinier-reg}, we deduce that
    \begin{align*}
    4 \Big(\Big(\Delta_k\big(
    y^{k/2}
    \theta(\cdot{,}\varphi)\big),
    y^{k/2}
    f\Big) \Big)_{\Pet}
    &=
    4 \Big(\Big(
    y^{k/2}
    \theta(\cdot{,}\varphi),\Delta_k\big(
    y^{k/2}
    f\big) \Big)\Big)_{\Pet}
    =
    dk(2-k)\Lambda_\varphi(f),
    \end{align*}
    where for the last equality we used that~$
    y^{k/2}f$ is an eigenform of~$\Delta_k$ with eigenvalue~$d\frac{k}{2}(1-\frac{k}{2})$, see Remark~\ref{rem-modformlaplacian}.
\end{proof}
Recall that if~$X_K$ is the Shimura variety~\eqref{eq;SHimvar}, then we denote by~$Y_K=\Gamma_K\backslash\domain^+$ the connected component of~$X_K$ arising from~$\Gamma_K=H(F)^+\cap K$.
Let~$H^{2d}(X_K,\CC)$ be the degree~$2d$ de Rham cohomology group of~$X_K$.
\begin{theorem}\label{cor;applicinL2}
    The Kudla--Millson lift~$\Lambda_{\varphiKM}=\Lambda_{\mathrm{KM}}^L$ of a Hilbert cusp form is a harmonic differential form on the orthogonal Shimura variety~$X_K$.
    If~$n>3$, then it is~$L^2$ and therefore
    \[
    \dim H^{2d}(X_K,\CC) \geq\dim \Lambda_{\varphiKM}(S_{k,L}).
    \]
    The same is true if~$X_K$ is replaced with its connected component~$Y_K$ and~$\Lambda_{\varphiKM}$ is restricted to~$Y_K$.
\end{theorem}
\begin{proof}
    The harmonicity of the lift follows from Theorem~\ref{thm:omegatodeltak}, since the Kudla--Millson Schwartz function is~$H(F_\infty)$-invariant and of parallel weight~$k=(1+n/2,\dots,1+n/2)$.
    Assume now that~$n>3$.
    With the same procedure as~\cite[Proposition~4.1]{bruinier-funke-inj}, which boils down to the Weil convergence criterion~\cite{weil-convcrit}, one can show that the lift is square-integrable.
By~$L^2$-Hodge theory, the $L^2$-cohomology group of degree~$m$ on~$Y_K$ is isomorphic to the space of harmonic $m$-forms on~$Y_K$, for every~$m$ \cite{BG83}.
Then the bound on cohomology follows from the fact that the~$L^2$-cohomology $H^r_{(2)}(Y_K,\CC)$ injects into~$H^r(Y_K,\CC)$ whenever~$r\leq c$, where~$c$ is the codimension of the singular locus of the Baily--Borel compactification~$\overline{Y_K}$ \cite[Section~5]{harriszucker}. 
Since~$c\geq \mathrm{codim} (\overline{Y_K}\setminus Y_K)\geq d(n-1)$, we have the aformentioned injection with $r=2d$ if~$n\geq 3$.
\end{proof}

\appendix
\section{Borcherds' formalism and infinitesimal actions}\label{appendix:Ftransf}
In this appendix we give details on how to compute the partial Fourier transforms provided in Lemma~\ref{lemma; FtransfKMschwartz}, the description over the tube domain model~$\tubedom$ of Borcherds' formalism and some ancillary results of differential operators on theta functions.

\subsection{A hint of the proof of Lemma~\ref{lemma; FtransfKMschwartz}}\label{sec:aphintproof}
    Recall that~$\varphiKM$ is defined as a product of Schwartz functions~$\varphiKM^{(i)}\in\mathcal{S}(V_i)\otimes\mathcal{Z}^2(D_i)$ for~$i=1,\dots,d$.
    Therefore,~$\mathcal{F}_\infty(\varphiKM)$ splits into a product of partial Fourier transforms of the~$\varphiKM^{(i)}$.
    We may then assume without loss of generality that~$F=\QQ$.
    
	Write~$\genvec=\sum_j \coeff_j\basevec_j\in V(\RR)$ over the pseudo-orthonormal basis~$(\basevec_j)_j$ of~$V(\RR)$ as in Section~\ref{sec:prelim}.
	Recall that if~$h_\infty\in P(\RR)$ is such that~$hz_0=z$, then
	\[
	\mathcal{F}_\infty(\varphiKM(z))(x_0,\eta_1,\eta_2)
	=
	\sum_{\indone,\indtwo=1}^n
	\mathcal{F}_\infty\big(\Gpol_{(\indone,\indtwo)}\varphi_0\circ h_\infty^{-1})(x_0,\eta_1,\eta_2)\cdot (h_\infty^{-1})^*(\omega_{\indone,\indtwo}).
	\]

	We rewrite the polynomials~$\Gpol_{(\alpha,\alpha')}$ with respect to the coordinates~$x_1,x_2$ of~$\ell'(\RR)\oplus\ell(\RR)$ and~$y_2,\dots,y_{n+1}$ of~$V_0(\RR)$ as 
	\[
	\Gpol_{(\alpha,\alpha')}(\genvec)
	=
	\begin{cases}
	(x_{1}+x_{2})^2-\frac{1}{2\pi} & \text{if~$\alpha=\alpha'=1$,}
	\\
    2y_{\alpha}^2 - \frac{1}{2\pi} & \text{if~${\alpha}={\alpha'}>1$.}
    \\
	\sqrt{2}(x_{1}+x_{2})y_{\alpha'} & \text{if~$\alpha=1 <\alpha'$,}
	\\
    \sqrt{2}(x_{1}+x_{2})y_{\alpha} & \text{if~$\alpha'=1 <\alpha$,}
	\\
	2y_{\alpha} y_{\alpha'} & \text{if~${\alpha}\neq{\alpha'}$ and~${\alpha},{\alpha'}>1$,}
	\end{cases}
	\]
	Note that~$\Gpol_{(\indone,\indtwo)}=\Gpol_{(\indtwo,\indone)}$ and that~$\Gpol_{(\indone,\indtwo)}$ depends on the variable of integration~$x_2$ only if~$\indone=1$ or~$\indtwo=1$.
	Let~$\varphi_{\RR}$ be the standard Gaussian on~$\RR$, and let~$\varphi_{j,\RR}(x_0,\eta_1,\eta_2)\coloneqq \varphi_\RR(\eta_j)$.
	One can show that 
	$\mathcal{F}_\infty\big(
	\Gpol_{(\indone,\indtwo)}\varphi_0
	\big)
	(x_0,\eta_1,\eta_2)
	=
	\varphi_{0}(x_0)
	\varphi_{\RR}(\eta_1)
	\mathcal{F}_\infty\big(
	\Gpol_{(\indone,\indtwo)}\varphi_{2,\RR}
	\big)
	(x_0,\eta_1,\eta_2)$
	and
	\begin{align}\label{eq;PFTz0}
	\mathcal{F}_\infty\big(
	\Gpol_{(\indone,\indtwo)}\varphi_{2,\RR}
	\big)
	(x_0,\eta_1,\eta_2)=\begin{cases}
	\Big(\Big(\eta_{1}^2 - \frac{1}{2\pi}\Big)
	-
	2i\eta_{1}\eta_{2}
	-
	\Big(\eta_{2}^2 - \frac{1}{2\pi}\Big)
	\Big)
	\varphi_{\RR}(\eta_2)
	& \text{if~$\indone=\indtwo=1$,}
	\\
	\sqrt{2}y_\indtwo
	(\eta_{1} - i\eta_{2})
	\varphi_{\RR}(\eta_2)
	&\text{if~$\indone=1 <\indtwo$,}
	\\
	\Gpol_{(\indone,\indtwo)}(x_0)\varphi_{\RR}(\eta_2)
	& \text{if~$\indone,\indtwo>1$.}
	\end{cases}
	\end{align}

\subsection{Borcherds' formalism over the tube domain model}\label{sec;appborc}

In this section we explain how to describe Borcherds' formalism, in particular the standard Gaussian and the decomposition~\eqref{eq;decomppolasinbor} of homogeneous polynomials, in terms of explicit functions on the tube domain model~$\tubedom$.
Along the way, we will prove Lemma~\ref{lemma:lemtbcfromunf1}.

We begin with a rewriting of the Gaussian~$\varphi_0$, see~\eqref{eq;stdgaus+Qab}.
If~$x_0=\sum_{j=2}^{n+1}y_je_j\in V_0(\RR)$, then
\begin{align*}
    \varphi_0(x_0)&=
    \exp\big(-\pi x_0^2+2\pi (x_0|_{\RR e_{n+1}})^2\big),
\end{align*}
where~$x_0|_W$ denotes the orthogonal projection of $x_0$ to some subspace~$W\subset V_0$.
Now we compose with~$h_Y^{-1}$. Since~$x_0|_{\RR Y}=\frac{(x_0,Y)}{Y^2}Y$, we deduce that
\begin{align*}
    \varphi_0(h_Y^{-1}(x_0))&=
    \exp\Big(-\pi \underbrace{(h_Y^{-1}x_0)^2}_{=x_0^2}+2\pi \underbrace{((h_Y^{-1}x_0)|_{\RR e_{n+1}})^2}_{=(x_0|_{h_Y(\RR e_{n+1})})^2}\Big)
    =
    \exp\Big(-\pi x_0^2 + 2\pi\frac{(x_0,Y)^2}{Y^2}\Big).
\end{align*}
This shows~\ref{item;lemtbcfromunf10} of Lemma~\ref{lemma:lemtbcfromunf1}.

We now describe the value~$(h^{-1})^\#v$, with~$v\in V(\RR)$, over the tube domain~$\tubedom$.
Here $h\in H(\RR)$ is any isometry mapping~$Z_0$ to~$Z\in \tubedom$.
Note that~$(h^{-1})^\#v=h^{-1}(v-v_{\RR \genU_{z^\perp}}-v_{\RR \genU_z})$.
By~\cite[Lemma~4.1]{zuffetti_unfolding1} we know that
the orthogonal projection of~$v$ to the line~$\RR \genU_z$ is
\begin{align*}
    v_{\RR \genU_z}
    &=
    \frac{(v,\genU_z)}{\genU_z^2}u_z
    =
    \frac{(v,X+\genUU+(q(Y)-q(X))\genU)}{Y^2}\cdot (X+\genUU+(q(Y)-q(X))\genU).
\end{align*}
Similarly, for~$v_{\RR \genU_{z^\perp}}$ we have
\begin{align*}
    v_{\RR \genU_{z^\perp}}
    &=
    \frac{(v,\genU_{z^\perp})}{\genU_{z^\perp}^2}\genU_{z^\perp}
    =
    -\frac{(v,-X-\genUU+(q(X)+q(Y))\genU)}{Y^2}\cdot(-X-\genUU+(q(X)+q(Y))\genU).
\end{align*}
It is then easy to show that
\begin{align*}
    (h^{-1})^\#v
    &=
    h^{-1}\big(v-(v,X+\genUU-X^2\genU)\genU - (v,\genU)X-(v,\genU)\genUU\big).
\end{align*}

    We now rewrite of~$(h^{-1})^\#v$ over the split~$V(\RR)=\ell\oplus V_0(\RR)\oplus \ell'$.
    Suppose that $v=x_1\genUU+x_0+x_2\genU$ for some~$x_1,x_2\in\RR$ and~$x_0\in V_0(\RR)$, so that~$v=(x_2,x_0,x_1)^t$.
    Then
    \begin{align}\label{action of h inverse}
        (h^{-1})^\#v&=
        h^{-1}\left(\begin{smallmatrix}
            x_2-(v,X+\genUU-X^2\genU) \\
            x_0-(v,\genU)X \\
            x_1-(v,\genU)
        \end{smallmatrix}\right).
    \end{align}
    If~$h=h_Z$ lies in the parabolic~$P(\RR)$ stabilizing~$\ell=\RR\genU$, then
    \begin{equation}\label{eq;indcF=QQ}
        (h_Z^{-1})^\#v
        =
        \left(\begin{smallmatrix}
        \lvert q(Y)\rvert^{-1/2}(x_2+(v,-X-\genUU+q(X)\genU)+(X,x_0-x_1X/2)\\
        h_Y^{-1}(x_0-x_1X)\\
        \lvert q(Y)\rvert^{1/2}(x_1-(v,\genU))
        \end{smallmatrix}\right).
    \end{equation}
    It is enough to replace~$x_1=(v,\genU)$ and~$x_2=(v,\genUU)$ in~\eqref{eq;indcF=QQ} to deduce ~\ref{item;lemtbcfromunf11} of Lemma~\ref{lemma:lemtbcfromunf1}.


    We now describe the decomposition~\eqref{eq;decomppolasinbor} of the polynomials~$\mathcal{P}_{(\indone,\indtwo)}$ in terms of the tube domain model~$\tubedom$.
    By~\cite[Lemma 3.6]{zuffetti_unfolding1}, since~$\genU_{z^\perp}^2=-1/Y^2$ and~$h_Z^{-1}(\genU)=\lvert q(Y)\rvert^{-1/2}\genU$, one can show that
    \begin{align*}
            &\mathcal{P}_{(\indone,\indtwo),(h_Z^{-1})^\#,r,0}((h^{-1}_Z)^\#v)\\
            &
            =
            \begin{cases}
                2Y^4(\lvert q(Y)\rvert^{-1/2}\genU,e_\indone)(\lvert q(Y)\rvert^{-1/2}\genU,e_\indtwo) & \text{if $r=2$},
                \\
                -2Y^2(\lvert q(Y)\rvert^{-1/2}\genU,e_\indone)((h^{-1}_Z)^\#v,e_\indtwo) - 2Y^2(\lvert q(Y)\rvert^{-1/2}\genU,e_\indtwo)((h^{-1}_Z)^\#v,e_\indone) & \text{if $r=1$},
                \\
                2((h^{-1}_Z)^\#v,e_\indone)((h^{-1}_Z)^\#v,e_\indtwo) & \text{if $r=0$.}
            \end{cases}
        \end{align*}
        An easy calculation using~\ref{item;lemtbcfromunf11} of Lemma~\ref{lemma:lemtbcfromunf1} shows that the previous description of $\mathcal{P}_{(\indone,\indtwo),(h_Z^{-1})^\#,r,0}((h^{-1}_Z)^\#v)$ boils down to~\ref{item;lemtbcfromunf12}.
    
    Now we compute the image of such polynomials under the operator $\exp\big(-\frac{\Delta}{8\pi y}\big)$.
    Recall that~$\mathcal{P}_{(\indone,\indtwo),(h_Z^{-1})^\#,0,0}$ is a polynomial on~$(h_Z^{-1})^\#(V(\RR))=V_0(\RR)$ by Lemma~\ref{lemma:lemtbcfromunf1} \ref{item;lemtbcfromunf11}.
    If~$a=\sum_{j=1}^{n+2}y_je_j\in V_0(\RR)$, then by~\cite[Lemma~3.6]{zuffetti_unfolding1} we know that if~$\alpha>1$, then $\mathcal{P}_{(\indone,\indone),(h_Z^{-1})^\#,0,0}(a)=2y_\indone^2$, but for~$\alpha=1$ that polynomial vanishes since~$\genU\perp V_0(\RR)$.
    Hence
    \begin{align*}
        -\frac{\Delta}{8\pi y}\Big(\mathcal{P}_{(\indone,\indone),(h_Z^{-1})^\#,0,0}\Big)
        =
        \begin{cases}
        -\frac{1}{2\pi y} & \text{if $\alpha>1$,}
        \\
        0 & \text{if $\alpha=1$.}
        \end{cases}
    \end{align*}
    In fact, the only non-trivial derivatives appear when $r=0$ and~$\alpha=\indtwo$, since only in these cases~$\mathcal{P}_{(\alpha,\indtwo),(h^{-1}_Z)^\#,r,0}$ is a polynomial containing a power~$y_j^2$ for some~$j$.
    This implies Lemma~\ref{lemma:lemtbcfromunf1}~\ref{item;lemtbcfromunf13}.

\subsection{Differential operators on theta functions}\label{sec:append:diff}


The following result, whose proof is standard and hence omitted, illustrates how to differentiate under~$d\omega$ the theta function with respect to the variables~$g$ and~$z$.

Recall
that the \emph{right regular actions} of~$G'_{F_\infty}$ and~$H(F_\infty)$ on respectively~$\mathcal{C}^{\infty}(G'_{F_\infty})$ and~$\mathcal{A}^{*}(\domain)$ are defined as
\begin{align*}
    (\rho_G(g_0)f)(g) &\coloneqq f(gg_0) \qquad\text{$f\in \mathcal{C}^\infty(G'_{F_\infty})$, $g,g_0\in G'_{F_\infty}$,}\\
    \rho_H(h)\alpha &\coloneqq h^*\alpha\qquad \text{$\alpha\in\mathcal{A}^*(\domain)$, $h\in H(F_\infty)$.}
\end{align*}
\begin{lemma}\label{lemma;rightgesss?}
Let~$\varphi\in\mathcal{S}(V_\infty)\otimes\mathcal{A}^{2d}(\domain)$.
    Then
    \[
    \big(d\rho_G(X)\theta(\cdot,\varphi)\big)(g)
    =
    \theta\big(g,d\omega(X)\varphi\big)
    \qquad\text{for all~$X\in U(\mathfrak{g})$.}
    \]
    If furthermore~$\varphi$ is~$H(F_\infty)$-invariant, then
    \[
    d\rho_H(Y)\theta(g,\varphi)
    =
    \theta\big(g,d\omega(Y)\varphi\big)
    \qquad\text{for all~$Y\in U(\mathfrak{h})$.}
    \]
\end{lemma}

We now focus on the Casimir elements of~$U(\mathfrak{g})$ and~$U(\mathfrak{h})$, explaining how they act on theta functions.
Let~$\varphi\in[\mathcal{S}(V_\infty)\otimes\mathcal{A}^{2d}(\domain)]^{H(F_\infty)}$.
By Lemma~\ref{lemma;rightgesss?} and Theorem~\ref{thm;howe} we may compute
\begin{equation}\label{eq;casimirsontheta}
    \begin{split}
    d\rho_H(C_H)\theta(g,\varphi)
     &=
     \theta(g,d\omega(C_H)\varphi)
     =
     \theta(g,d\omega(-4C_G)\varphi + \xi\varphi)
     \\
    &=-4\cdot d\rho_G(C_G)\theta(g,\varphi) + \xi\theta(g,\varphi).
    \end{split}
\end{equation}

We now rewrite everything in terms of vector-valued Hilbert modular forms.
Assume that~$\varphi$ has parallel weight~$k\in\ZZ+\frac{n}{2}$, and let~$\theta(\tau,\varphi)$ the~$\CC[L^\vee/L]\otimes\mathcal{A}^{2d}(\domain)$-valued theta function as in~\eqref{eq;newthetatau}.
By~\cite[Proposition~2.2.5]{bump} the operator~$d\rho(C_{G,r})$ acts on~$y^{k/2}\theta(\tau,\varphi)$ as the weight~$k$ hyperbolic Laplacian on the~$r$-th copy of~$\HH$ in~$\HH^d$, namely as
\[
\Delta_{k,r}=-y_r^2\Big(\frac{\partial^2}{\partial x_r^2} + \frac{\partial^2}{\partial y_r^2}\Big) + ik y_r\frac{\partial}{\partial x_r}.
\]
Therefore, $d\rho(C_G)$ acts on $y^{k/2}\theta(\tau,\varphi)$ as
\[
\Delta_k\coloneqq \sum_{r=1}^d\Delta_{k,r}.
\]
\begin{remark}\label{rem-modformlaplacian}
    Let~$f\in M_{k,L}$ be of parallel weight~$k$.
    Similarly as in~\cite[Exercise~2.1.7, (a)]{bump}, we have
    \[
\Delta_k \big(y^{k/2}f(\tau)\big) = d\frac{k}{2}\Big(1-\frac{k}{2}\Big)y^{k/2}f(\tau),
\]
namely~$y^{k/2}f$ is an eigenfunction of~$\Delta_k$ of eigenvalue~$d\frac{k}{2}\big(1-\frac{k}{2}\big)$.
\end{remark}

Considering the theta function as a differential form on~$\domain$, by~\cite[p.~451, Ex.\ 5]{helgason} the operator~$d\rho_H(C_H)$ acts as~$-\Omega$.
Summarizing, we deduce the following result.
\begin{corollary}\label{cor:omtodelfortheta}
    Let~$\varphi\in [\mathcal{S}(V_\infty)\otimes\mathcal{A}^{2d}(\domain)]^{H(F_\infty)}$ be of parallel weight~$k$.
    Then
    \[
    -\Omega \Big(
    y^{k/2}
    \theta(\tau,\varphi)\Big) = -4\Delta_k\Big(
    y^{k/2}\theta(\tau,\varphi) \Big) + \xi 
    y^{k/2}
    \theta(\tau,\varphi).
    \]
\end{corollary}

\printbibliography

\end{document}